\documentclass[10.5pt,a4paper]{article}

\usepackage{graphicx,latexsym,euscript,makeidx,color,bm}
\usepackage{amsmath,amsfonts,amssymb,amsthm,thmtools,mathtools,mathrsfs,enumerate}
\usepackage[colorlinks,linkcolor=blue,anchorcolor=green,citecolor=red]{hyperref}

\usepackage{tikz}
\usetikzlibrary{fit,calc,positioning} 
\tikzstyle{Block}=[rectangle,minimum width=3cm,minimum height=1cm,text centered,text width=5.2cm,draw=black]
\tikzstyle{Implication}=[rectangle,minimum width=3cm,minimum height=1cm,text centered,text width=5.2cm]
\tikzstyle{jian}=[<->, >=stealth]

\usepackage{geometry}
\allowdisplaybreaks[4]

  \def\cA{{\cal A}}  
  \def\cB{{\cal B}}  
  \def\cC{{\cal C}}  
    
\def\dbE{\mathbb{E}}    
\def\dbF{\mathbb{F}}  \def\cF{{\cal F}}  
  \def\cG{{\cal G}}  
\def\dbH{\mathbb{H}}    
    
  \def\cJ{{\cal J}}  
    
  \def\cL{{\cal L}}  
  \def\cM{{\cal M}}  
  \def\cN{{\cal N}}  
    
\def\dbP{\mathbb{P}}    
  \def\cQ{{\cal Q}}  
\def\dbR{\mathbb{R}}  \def\cR{{\cal R}}  
\def\dbS{\mathbb{S}}  \def\cS{{\cal S}}  
    
  \def\cU{{\cal U}}

  \def\cX{{\cal X}}

      \def\lt{\left}       \def\hb{\hbox}
\def\ms{\medskip}        \def\rt{\right}      \def\ae{\hb{a.e.}}
\def\bs{\bigskip}        \def\lan{\langle}    
   \def\ran{\rangle}    \def\tr{\hb{tr$\,$}}
\def\ts{\textstyle}         
\def\no{\noindent}          
\def\hp{\hphantom}         
\def\nn{\nonumber}         
\def\rf{\eqref}            
\def\cd{\cdot}             
\def\deq{\triangleq}     \def\({\Big (}       \def\ba{\begin{aligned}}
\def\les{\leqslant}      \def\){\Big )}       \def\ea{\end{aligned}}
\def\ges{\geqslant}      \def\[{\Big[}        \def\bel{\begin{equation}\label}
\def\ti{\tilde}          \def\]{\Big]}        \def\ee{\end{equation}}
      \def\q{\quad}        
\def\h{\widehat}         \def\qq{\qquad}      
                                              
\def\a{\alpha}  \def\G{\Gamma}      \def\Om{\Omega}  
\def\b{\beta}   \def\D{\Delta}   \def\d{\delta}        
   \def\Th{\Theta}    \def\Si{\Sigma}  \def\si{\sigma}
\def\f{\varphi}   \def\l{\lambda}        
    \def\i{\infty}      

\newtheoremstyle{thry}
{}      
{}      
{\sl}   
{}      
{\bf}   
{.}     
{.5em}  
{}      

\theoremstyle{thry}

\newtheorem{theorem}{Theorem}[section]
\newtheorem{proposition}[theorem]{Proposition}

\newtheorem{lemma}[theorem]{Lemma}

\theoremstyle{definition}
\newtheorem{definition}[theorem]{Definition}
\newtheorem{example}[theorem]{Example}

\newtheorem{taggedassumptionx}{}
\newenvironment{taggedassumption}[1]
 {\renewcommand\thetaggedassumptionx{#1}\taggedassumptionx}
 {\endtaggedassumptionx}

\theoremstyle{remark}
\newtheorem{remark}[theorem]{Remark}

\makeatletter
   
   \@addtoreset{equation}{section}
   \newcommand{\setword}[2]{%
   \phantomsection
   #1\def\@currentlabel{\unexpanded{#1}}\label{#2}%
   }
\makeatother

\begin{document}

\title{\bf Zero-Sum Stackelberg Stochastic Linear-Quadratic Differential Games}

\author{Jingrui Sun\thanks{Department of Mathematics, Southern University of Science and Technology,
                           Shenzhen 518055, China (Email: {\tt sunjr@sustech.edu.cn}).
                           This author is supported by NSFC grant 11901280 and
                           Guangdong Basic and Applied Basic Research Foundation 2021A1515010031.}~~~~
       Hanxiao Wang\thanks{Corresponding author.
       Department of Mathematics, National University of Singapore,
                           Singapore 119076, Singapore (Email: {\tt hxwang14@fudan.edu.cn}).
                           This author is supported by Singapore MOE AcRF Grant R-146-000-271-112.}~~~~
      Jiaqiang Wen\thanks{Department of Mathematics, Southern University of Science and Technology,
                           Shenzhen 518055, China (Email: {\tt wenjq@sustech.edu.cn}).
                           This author is supported by SUSTech start-up fund Y01286233.}}

\maketitle

\no{\bf Abstract.}
The paper is concerned with a zero-sum Stackelberg stochastic linear-quadratic (LQ, for short)
differential game over finite horizons.
Under a fairly weak condition, the Stackelberg equilibrium is explicitly obtained by first solving
a forward stochastic LQ optimal control problem (SLQ problem, for short) and then a backward SLQ problem.
Two Riccati equations are derived for constructing the Stackelberg equilibrium.
An interesting finding is that the difference of these two Riccati equations coincides with the Riccati
equation associated with the zero-sum Nash stochastic LQ differential game,
which implies that the Stackelberg equilibrium and the Nash equilibrium are actually identical.
Consequently, the Stackelberg equilibrium admits a linear state feedback representation,
and the Nash game can be solved in a leader-follower manner.


\ms
\no{\bf Keywords.} \rm
Stochastic differential game, Stackelberg equilibrium, linear-quadratic, two-person, zero-sum,
Nash equilibrium, Riccati equation, closed-loop representation.

\ms
\no{\bf AMS subject classifications.} 91A15, 93E20, 49N10, 49N70.

\section{Introduction}\label{Sec:Intro}
Let $(\Om,\cF,\dbP)$ be a complete probability space, $W$ a one-dimensional standard Brownian motion,
and $\dbF\equiv\{\cF_t\}_{t\ges0}$ the usual augmentation of the natural filtration generated by $W$.
For a given {\it initial state} $x\in\dbR^n$, consider the following controlled linear stochastic differential
equation (SDE, for short) on a finite horizon $[0,T]$:
\bel{state}\left\{\begin{aligned}
   dX(s) &= [A(s)X(s) +B_1(s)u_1(s) +B_2(s)u_2(s)]ds\\
         &\hp{=\ } +[C(s)X(s) +D_1(s)u_1(s)+ D_2(s)u_2(s)]dW(s), \\
    X(0) &= x,
\end{aligned}\right.\ee
where $A,C:[0,T]\to\dbR^{n\times n}$ and $B_i, D_i:[0,T]\to\dbR^{n\times m_i}$ ($i=1,2$),
called the {\it coefficients} of the {\it state equation} \rf{state}, are given deterministic functions.
The problem involves two players with opposing aims. Each player can affect the evolution of the system \rf{state}
by selecting his/her own {\it control}.
In the above, the process $u_i$ ($i=1,2$) represents the control of Player $i$, which belongs to the following space:
$$\ts \cU_i = \Big\{\f:[0,T]\times\Om\to\dbR^{m_i} \bigm| \f\hb{~is $\dbF$-progressively measurable, and}
~\dbE\int^T_0|\f(s)|^2ds<\i\Big\}. $$
The solution $X(\cd)\equiv X(\cd\,;x,u_1,u_2)$ of \rf{state} is called
the {\it state process} corresponding to $x$ and $(u_1,u_2)$.
The criterion for the performance of $u_1$ and $u_2$ is given by the following quadratic functional:
\bel{cost}
J(x;u_1,u_2) = \dbE\lt\{\lan GX(T),X(T)\ran
+\int_0^T\[\lan Q(s)X(s),X(s)\ran+ {\ts\sum_{i=1}^2}\lan R_i(s) u_i(s),u_i(s)\ran\]ds  \rt\},
\ee
where $G$ is an $n\times n$ symmetric matrix; $Q:[0,T]\to\dbR^{n\times n}$
and $R_{i}:[0,T]\to\dbR^{m_i\times m_i}$ ($i=1,2$) are deterministic, symmetric matrix-valued functions.

\ms

In our problem, Player $2$ is the leader, who announces his/her control $u_2$ first, and Player $1$ is the follower,
who chooses his/her control accordingly.
The criterion functional $J(x;u_1,u_2)$ is regarded as the loss of Player $1$ and the gain of Player $2$.
So whatever the leader announces, the follower will play optimally; that is, Player $1$ will select a control
$\bar u_1(\cd\,;u_2,x)$ (depending on the control $u_2$ announced by the leader as well as the initial state $x$)
such that $J(x;u_1,u_2)$ is minimized.
Knowing this the leader will choose a $\bar u_2$ a priori so that $J(x;\bar u_1(\cd\,;u_2,x),u_2)$ is maximized.
Such a game is referred to as a {\it two-person Stackelberg stochastic linear-quadratic (LQ, for short) differential game}
(denoted by Problem (SG)), in memory of Stackelberg's pioneering contribution in this field.
The main objective of the two players is to find the {\it Stackelberg equilibrium}, mathematically defined as follows.

\begin{definition}\label{def:S-junheng}
A control pair $(\bar u_1,\bar u_2)\in\cU_1\times\cU_2$ is called a {\it Stackelberg equilibrium} for the initial state $x$ if
$$ \inf_{u_1\in\cU_1}J(x;u_1,\bar u_2) = J(x;\bar u_1,\bar u_2) = \sup_{u_2\in\cU_2}\inf_{u_1\in\cU_1}J(x;u_1,u_2). $$
\end{definition}

As mentioned earlier, when playing the game, the control selected by the follower depends on the initial state and
the control announced by the leader.
This means that the optimal control of the follower is a mapping from $\cU_2\times\dbR^n$ into $\cU_1$,
which is usually referred as to an {\it Elliot--Kalton strategy}.
Thus, the players may use the optimal control-strategy pair as an alternative solution to the game.
More precisely, we have the following definition.

\begin{definition}\label{def:S-junheng*}
Let $\G_1$ be the set of all Elliot--Kalton strategies for Player $1$.
A control-strategy pair $(\bar \a_1,\bar u_2)\in\G_1\times\cU_2$ is said to be {\it optimal} for the initial state $x$ if
\begin{align*}
 J(x;\bar\a_1(u_2,x),u_2) &=\inf_{u_1\in\cU_1}J(x;u_1,u_2), \q \forall u_2\in\cU_2, \\
 J(x;\bar\a_1(\bar u_2,x),\bar u_2) &=\sup_{u_2\in\cU_2}J(x;\bar\a_1(u_2,x), u_2).
\end{align*}
\end{definition}

Comparing \autoref{def:S-junheng*} with \autoref{def:S-junheng}, it is not hard to see that
the outcome $(\bar u_1,\bar u_2)\equiv(\bar\a_1(\bar u_2,x),\bar u_2)$ of an optimal control-strategy
pair is a Stackelberg equilibrium.

\ms

Since the pioneering work \cite{Stackelberg1934} by Stackelberg, the theory of Stackelberg games
has been widely used in  economics, finance, and engineering; such as the famous principal--agent
model (see, for example, \cite{Bensoussan2018,Cvitanic-Zhang2012}).
The Stackelberg  stochastic LQ differential  game was initially studied by Bagchi and Basar \cite{Bagchi-Basar1981}.
In 2002, a general framework was formulated by Yong \cite{Yong2002}, in which the leader's problem was described
as an LQ optimal control problem for forward-backward SDEs.
By a decoupling method, Yong showed that the open-loop solution can be represented as a state feedback form,
provided the associated stochastic Riccati equation is solvable.
From then on, there has been extensive research on  Stackelberg stochastic LQ game problems.
For example, in \cite{Bensoussan-Chen-Sethi2015} Bensoussan, Chen, and Sethi established the maximum
principle for Stackelberg games;
Shi, Wang, and Xiong \cite{Shi-Wang-Xiong2016} investigated a Stackelberg stochastic LQ differential
game with asymmetric information;
Moon \cite{Moon2021} studied the case with jump-diffusion systems;
Bensoussan {\it et al.} \cite{Bensoussan-Chau-Lai-Yam2017} considered a mean-field problem with
state and control delays;
Li and Yu \cite{Li-Yu2018} characterized the unique equilibrium of a nonzero-sum Stackelberg LQ game
with multilevel hierarchy in a closed form;
and Moon and Yang \cite{Moon-Yang2020} discussed the time-consistent open-loop solutions
for time-inconsistent Stackelberg LQ games.

\ms

In the literature, it is often assumed that the associated Riccati equations are solvable so that the Stackelberg
equilibrium can be constructed explicitly. However, such an assumption seems too strong in certain situations
because the solvability of the Riccati equations is merely sufficient, but not necessary for the existence of
a Stackelberg equilibrium.
Since the solvability of the Riccati equations itself is very difficult, solving Problem (SG) in a general framework
is more challenging.
The first goal of our paper is to overcome this difficulty in the zero-sum case
and then establish a general approach for finding the Stackelberg equilibrium of Problem (SG) by generalizing
the recent works \cite{Sun-Li-Yong2016,Sun-Wu-Xiong2021} on indefinite stochastic LQ optimal control problems
to the nonhomogenous case and making some new observations.
This can be regarded as one of the main contributions in this paper.

\ms

Another important kind of zero-sum stochastic LQ differential games is the so-called {\it Nash game},
in which both players announce their decisions simultaneously (see, for example,
\cite{Mou-Yong2006,Sun-Yong2014,Sun-Yong20202}).
In a Nash game, the objective of the players is to find a {\it saddle point} $(u_1^*,u_2^*)$
(also called a {\it Nash equilibrium}), defined by
\begin{equation}\label{saddle-point-inequality}
  J(x;u_1^*,u_2) \les J(x;u_1^*,u_2^*) \les J(x;u_1,u_2^*), \q\forall (u_1,u_2)\in\cU_1\times\cU_2.
\end{equation}
Such a pair (if exists) is the best choice for both players in the sense that no player can benefit
by changing their own control.
For simplicity, we shall denote the Nash game by Problem (NG).

\ms

Note that though the players in Problem (SG) have opposite objects,
they still agree to make some cooperations,
because there is a hierarchical structure of decision making between the players.
However,  the players in Problem (NG)  are pure competitors as they are treated on an equal basis.
Thus in most of the literature, if not all,  Problems (NG) and (SG) are regarded as two different games.
In this paper, we shall compare Problem (SG) and Problem (NG) more carefully, and reveal an interesting fact:
{\it the Stackelberg equilibrium and the Nash equilibrium coincide} under a {\it uniform convexity-concavity condition}
((UCC) condition, for short).
This is another important contribution of the paper and is completely new in the literature.

\subsection{The main results}
As mentioned, the purpose of this paper is to develop a general approach for solving Problem (SG)
and to establish the connection between Problem (SG) and Problem (NG).
We now briefly list our ideas and main results as follows.
\begin{enumerate}[(i)]
\item
We consider first, for a fixed $u_2\in\cU_2$, the follower's problem, which we denote by Problem (FLQ).
By a result form Sun, Li, and Yong \cite{Sun-Li-Yong2016}, we know that Problem (FLQ) admits a unique open-loop
optimal control $\bar u_1$ of the form $\bar u_1=\bar \a_1(u_2,x)\equiv \Th_1 \bar X+v_1\in\cU_1$,
where the function $\Th$ and the process $v_1$ are determined by the associated Riccati equation
and the associated backward SDE (BSDE, for short), respectively.
Note that the state process $\bar X$ depends on the initial state $x$, and both $\bar X$ and $v_1$ depend on the
given control $u_2$. Thus, $\bar u_1$ is a functional of $u_2$ and $x$.

\item
Knowing the follower will use his/her best response $\bar u_1=\bar \a_1( u_2,x)$,
the leader's problem (denoted by Problem (LLQ)) is then to choose a $\bar u_2\in\cU_2$ to maximize the utility functional
$$ J(x;\bar \a_1(u_2,x),u_2)\equiv J(x; \Th_1 \bar X(u_2)+v_1(u_2),u_2). $$
A remarkable feature of the above functional is that it has an explicit representation independent of the forward
state process $\bar X$. Using this crucial observation,
we convert the leader's problem into a backward stochastic LQ optimal control problem.

\item
We develop some results on backward stochastic LQ optimal control problems with nonhomogeneous terms (see \autoref{Prop:backward}),
solve the backward control problem derived from the Stackelberg game (see \autoref{prop:leader}), and then verify the resulting
control pair is a Stackelberg equilibrium of the game (see \autoref{Thm:SOP-control}).
\end{enumerate}

The following is concerned with the connections between Problems (SG) and (NG).
\begin{enumerate}[(i)]\setcounter{enumi}{3}
\item
Under the (UCC) condition (i.e., \ref{ass:H3} and \ref{ass:H5}), we study the leader's problem
by a careful convexity analysis of the criterion functional (see \autoref{Prop:saddle-point-control})
and a closer investigation of backward stochastic LQ optimal control problems (see \autoref{Prop:backward1}).
We obtain the unique optimal control of Problem (LLQ) by solving a new Riccati equation (see \autoref{prop:leader1}),
in which the  auxiliary function introduced in \cite{Sun-Wu-Xiong2021} is removed.
Then we further show that the Stackelberg equilibrium of Problem (SG) admits a closed-loop representation
(see \autoref{Thm:SOP-control1}).

\item
We find an interesting fact: the solutions to the Riccati equations associated with
Problems (FLQ) and (LLQ) can be used to solve the Riccati equation derived in Sun \cite{Sun2020}
for finding the saddle point of Problem (NG) (see \autoref{Thm:Riccati}).
A key point of the proof is to build a bridge between the singular terms of these Riccati equations
(see \autoref{lem:inverse}), which can be regarded as the most technical part of the paper.
Moreover, \autoref{Thm:Riccati} generalizes the results of \cite{Sun2020} at least in two aspects:

\begin{itemize}
\item The well-posedness of the Riccati equation associated Problem (NG) is established
under a weaker assumption and with a new constructive method (see \autoref{remark-Riccati}).

\item An explicit relationship between the Riccati equations associated with Problems (SG), (FLQ) and (LLQ)
is established, which is interesting in its own right and new in the literature.
\end{itemize}

\item
We observe that the closed-loop systems of Problems (SG) and (NG) coincide (see \autoref{closed-system-SPG}),
from which we conclude that the Stackelberg equilibrium obtained in \autoref{Thm:SOP-control} and the unique
open-loop saddle point of Problem (NG) are identical (also see \autoref{Thm:saddle-point-control} for a direct proof).
This means that we can solve the Nash game in a leader-follower manner.
\end{enumerate}

The remainder of this paper is structured as follows.
In Subsection \ref{sect-literature}, we give a literature review on some closely related topics.
Section \ref{Sec:Preliminaries} collects some preliminary results that will be frequently used in the sequel.
In Section \ref{Sec:SOG}, the Stackelberg equilibrium of Problem (SG) is obtained by
solving a forward-backward stochastic LQ optimal control problem.
Section \ref{sec:Further-Analysis} is devoted to the closed-loop representation of the Stackelberg equilibrium
by some further analysis of backward stochastic LQ optimal control problems.
The connection between Problems (SG) and (NG) is established in Section \ref{sec:Connection},
and Section \ref{Sec:Conclusion} concludes the paper.
Some technical details are sketched in Appendix.

\subsection{Literature review on the related topics}\label{sect-literature}

The LQ control/game theory has occupied the center stage for research in control theory for a long history.
Since the purpose of the paper is not to make a lengthy survey on the literature,
we only list some closely related works here.
In Problem (SG), the follower's problem is a (forward) stochastic LQ optimal control problem.
We refer the reader to the books \cite[Chapter 6]{Yong-Zhou1999} and \cite{Sun-Yong20201} for a
detailed study of this subject.
In Problem (SG), the leader's problem is a backward stochastic LQ optimal control problem,
which was initially investigated by Lim and Zhou \cite{Lim-Zhou2001},
and then generalized by \cite{Li-Sun-Xiong2019,Sun-Wang2019,Sun-Wu-Xiong2021} to various cases.
The results obtained in Section \ref{Sec:SOG} benefit from the recent work of Sun, Wu, and Xiong \cite{Sun-Wu-Xiong2021} a lot.
However, to explore the connection between Problems (SG) and (NG) in Sections \ref{sec:Further-Analysis} and \ref{sec:Connection},
we still need to overcome some mathematical difficulties (see, for example, \autoref{Thm:Riccati})
and to make some more accurate observations (see, for example, \autoref{Prop:saddle-point-control} and \autoref{Thm:saddle-point-control}).
For more information and references on Problem (NG), we send the interested reader to the works
\cite{Zhang2005,Mou-Yong2006,Delfour2007,Delfour-Sbarba2009,Sun-Yong2014,Yu2015,Sun2020,Sun-Wang-Wu2021} and
the recent book \cite{Sun-Yong20202} by Sun and Yong.

\section{Preliminaries}\label{Sec:Preliminaries}

Throughout the paper,  $\dbR^{n\times m}$ denotes the Euclidean space consisting of $n\times m$ real matrices,
endowed with the Frobenius inner product $\lan M,N\ran=\tr[M^\top N]$,
where $M^\top$ is the transpose of $M$ and $\tr(M)$ is the trace of $M$.
The norm induced by  $\lan\cd\,,\cd\ran$ is denoted by $|\cd|$.
The identity matrix of size $n$ is denoted by $I_n$,
which is often simply written as $I$ when there is no confusion.
When $m=1$, we simply write $\dbR^{n\times 1}$ as $\dbR^n$.
Let $\dbS^n$  be the subspace of $\dbR^{n\times n}$ consisting of symmetric matrices
and $\dbS^n_+$ (resp., $\dbS^n_-$) be the subset of $\dbS^n$ consisting of positive (resp., negative) semidefinite matrices.
For $M,N\in\dbS^n$, we write $M\ges N$ (resp., $M>N$) if $M-N$ is positive
semidefinite (resp., positive definite).
For an $\dbS^n$-valued measurable function $F$ on $[0,T]$, we write
$$
\left\{
\begin{aligned}
& F \ges 0   &&\q \hb{if}\q  F(s)\ges 0,       &&\q \ae~s\in[0,T],\\
& F   >  0  &&\q \hb{if}\q F(s)   > 0,       &&\q \ae~s\in[0,T],\\
& F \gg  0  &&\q \hb{if}\q  F(s)\ges \d I_n,  &&\q \ae~s\in[0,T],~\hb{for some}~\d>0.
\end{aligned}\right.
$$
Moreover, we use $F\les 0$, $F<0$ and $F\ll 0$ to indicate that  $-F\ges 0$, $-F>0$ and $-F\gg0$, respectively.
If $F\gg 0$ (resp., $F\ll 0$), we say that $F$ is uniformly positive (resp., negative) definite.
For any Euclidean space $\dbH$ (which could be $\dbR^n$, $\dbR^{n\times m}$, $\dbS^n$, etc.),
we introduce the following spaces:
\begin{align*}
L^\i(0,T;\dbH)
&= \big\{\f:[0,T]\to\dbH~|~\f~\hb{is  essentially bounded}\big\};\\
L_{\cF_T}^2(\Om;\dbH)
  &= \big\{\xi:\Om\to\dbH~|~\xi~\hb{is $\cF_T$-measurable and}~\dbE[|\xi|^2]<\i\big\};\\
L_\dbF^2(0,T;\dbH)
&= \ts\Big\{\f:[0,T]\times\Om\to\dbH~|~\f\hbox{ is $\dbF$-progressively measurable},\\
&\qq\q \hb{and}~\dbE\int_0^T|\f(s)|^2ds<\i\Big\};\\
L_\dbF^2(\Om;C([0,T];\dbH))
&= \ts\Big\{\f:[0,T]\times\Om\to\dbH~|~\f\in\dbF~\hb{is continuous, $\dbF$-adapted, }\\
&\qq\q \hbox{and }  \dbE\[\sup_{0\les s\les T}|\f(s)|^2\]<\i\Big\}.
\end{align*}

To guarantee that Problem (SG) is well-posed, we assume that the coefficients of state equation
\rf{state} and the weighting matrices in quadratic functional \rf{cost} satisfy the following conditions.

\begin{taggedassumption}{(H1)}\label{ass:H1}
The coefficients of state equation \rf{state} satisfy
$$
A,C \in L^\i(0,T;\dbR^{n\times n}),
\q  B_i, D_i \in L^\i(0,T;\dbR^{n\times m_i});~ i=1,2.
$$
\end{taggedassumption}

\begin{taggedassumption}{(H2)}\label{ass:H2}
The weighting matrices in  quadratic functional \rf{cost} satisfy
$$
G\in\dbS^n,\q
Q\in  L^\i(0,T;\dbS^{n}),\q
 R_{i}\in L^\i(0,T;\dbS^{m_i});~i=1,2.
$$
\end{taggedassumption}

Let \ref{ass:H1} hold.
For any $x\in\dbR^n$ and $(u_1,u_2)\in\cU_1\times\cU_2$, by the standard results of SDEs,
state equation \rf{state} admits a unique solution $X\in L^2_{\dbF}(\Om;C([0,T];\dbR^n))$.
Then  under assumption \ref{ass:H2}, the quadratic functional \rf{cost} is well-defined
and thus Problem (SG) is well-posed.

\ms
As mentioned in the introduction section, the game with state equation \rf{state} and functional \rf{cost}
can be formulated as two different problems (i.e., Problem (SG) and Problem (NG)).
Recall \autoref{def:S-junheng} and \autoref{def:S-junheng*}, in which the notions of Stackelberg equilibria
and Elliot--Kalton strategies associated with Problem (SG) are introduced.
Now let us present an important notion of Problem (NG).

\begin{definition}\label{definition-saddle-point}
A control pair $(u_1^*,u_2^*)\in\cU_1\times\cU_2$ is called an {\it open-loop saddle point} (or {\it a Nash equilibrium})
of Problem (NG) for the initial state $x\in\dbR^n$ if
\bel{definition-saddle-points}
J(x;u^*_1,u_2)\les J(x;u^*_1,u^*_2)\les J(x;u_1,u^*_2),\q\forall (u_1,u_2)\in\cU_1\times\cU_2.
\ee
For any $x\in\dbR^n$, we call $V(x)$  a {\it value} of Problem (NG) at $x$ if
\bel{definition-VF}
V(x)=\inf_{u_1\in\cU_1}\sup_{u_2\in\cU_2}J(x;u_1,u_2)
=\sup_{u_2\in\cU_2}\inf_{u_1\in\cU_1}J(x;u_1,u_2).
\ee
\end{definition}

\begin{remark}
The value function $V$ of Problem (NG) is well-defined at $x\in\dbR^n$
only when the following inequality holds:
$$
\inf_{u_1\in\cU_1}\sup_{u_2\in\cU_2}J(x;u_1,u_2)
\les \sup_{u_2\in\cU_2}\inf_{u_1\in\cU_1}J(x;u_1,u_2).
$$
\end{remark}

\subsection{Backward  stochastic LQ optimal control problems with nonhomogeneous terms}
In this subsection, we shall generalize the results of indefinite backward stochastic LQ optimal control problems
obtained by Sun, Wu, and Xiong \cite{Sun-Wu-Xiong2021} to the case with nonhomogeneous terms.

\ms
For any given terminal state $\xi\in L^2_{\cF_T}(\Om;\dbR^n)$, consider the following controlled linear BSDE:
\bel{backward-system}\left\{\begin{aligned}
         dY(s) &=\big[\cA(s) Y(s) +\cB(s)u_2(s)+\cC(s)Z(s)+\si(s) \big]ds \\
         &\q+Z(s) dW(s), \q s\in[0,T],\\
          Y(T) &=\xi,
\end{aligned}\right.\ee
and the utility functional:
\begin{align}\label{utility-backward}
&U(\xi; u_2)= \dbE\bigg\{\int_0^T\big[\lan \cR u_2, u_2\ran+\lan \cQ Y,Y\ran+\lan \cN Z,Z\ran
+2\lan \cS_1 Y, u_2\ran+2\lan \cS_2 Z,u_2\ran+2\lan \cS_3 Y,Z\ran\big] ds \nn\\
&\qq\qq\qq +\lan \cG Y(0),Y(0)\ran+2\lan Y(0),g\ran \bigg\}.
\end{align}
The associated backward stochastic LQ optimal control problem
can be stated as follows.

\ms
\noindent
{\bf Problem (BLQ).} For any given terminal state $\xi\in L^2_{\cF_T}(\Om;\dbR^n)$,
find a control $\bar u_2 \in \mathcal{U}_2$ such that
\bel{sup-U}
U(\xi;\bar u_2)=\sup_{ u_2\in \mathcal{U}_2} U(\xi;u_2)\equiv \bar U(\xi).
\ee

In the following, we are going to find an optimal control of Problem (BLQ),
by similar  arguments to those employed in \cite[Theorem 6.3]{Sun-Wu-Xiong2021}.

\begin{taggedassumption}{(B1)}\label{ass:B3}\rm
The coefficients of state equation \rf{backward-system} and the weighting matrices in functional \rf{utility-backward} satisfy
\begin{align}
&\cA,\cC \in L^\i(0,T;\dbR^{n\times n}),
\q  \cB\in L^\i(0,T;\dbR^{n\times m_2}),\q\cR\in L^\i(0,T;\dbS^{m_2}),\nn\\
&\cQ,\cN\in  L^\i(0,T;\dbS^{n}),\q
\q \cS_i\in L^\i(0,T;\dbR^{m_2\times n});~i=1,2,\nn\\
& \cS_3 \in L^\i(0,T;\dbR^{n\times n}), \q\si\in L^2_{\dbF}(0,T;\dbR^n),\q \cG\in\dbS^n,\q g\in\dbR^n.\label{B1-1}
\end{align}
Moreover, there exists a constant $\l>0$ such that
\bel{uniform-concave-condition}
U_0(0;u_2)\les -\l \dbE\int_0^T |u_2(s)|^2ds,\q \forall u_2\in\cU_2,
\ee
where  $U_0(\xi;u)$ denotes the utility functional $U(\xi;u)$ with  $\si\equiv0$ and $g=0$.
\end{taggedassumption}

It is noteworthy that the condition \rf{uniform-concave-condition} implies
$
\cR \ll 0,
$
whose proof can be found in \cite[Corollary 5.3]{Sun-Wu-Xiong2021}.
Under \ref{ass:B3}, by \cite[Theorem 6.2]{Sun-Wu-Xiong2021} the following Riccati equation admits
a unique negative semidefinite solution $\Si^H\in C(0,T;\dbS^n_-)$:
\bel{BLQ-Riccati-Equation}
\left\{\begin{aligned}
&\dot{\Si}^H-\Si^H \cA^\top-\cA\Si^H+[\cB+\Si^H(\cS_1^H)^\top]\cR^{-1}[\cB^\top+\cS_1^H\Si^H]\\
&\q +[\cL+\Si^H(\cS_3^H)^\top][I+\Si^H\cN^H]^{-1}\Si^H[\cL^\top +\cS^H_3\Si^H]=0,\q t\in[0,T],\\
& \Si^H(T)=0,
\end{aligned}\right.
\ee
where
\begin{align}
\cL&= \cC-\cB\cR^{-1}\cS_2,\q\cN^H= \cN-\cS_2^\top\cR^{-1}\cS_2+H,\nn\\
\q\cS_1^H&=\cS_1+\cB^\top H, \q \cS^H_3= \cS_3-\cS_2^\top\cR^{-1}\cS_1+\cL^\top H,
\label{def-NH}
\end{align}
with the auxiliary function $H$ uniquely determined by the following ordinary differential equation (ODE, for short):
\bel{H-equation}
\left\{\begin{aligned}
&\dot{H}+H\cA+\cA^\top H+\cQ=0,
\q t\in[0,T],\\
& H(0)=-\cG.
\end{aligned}\right.
\ee
Moreover, the function $\h\Si^H$, defined by
\bel{hat-Si}
\h\Si^H\equiv I+\Si^H\cN^H,
\ee
is invertible with $(\h\Si^H)^{-1}\in L^\i(0,T;\dbR^{n\times n})$.
With the unique solution $\Si^H$ of \rf{BLQ-Riccati-Equation}, we introduce the following BSDE:
\bel{BLQ-BSDE}
\left\{\begin{aligned}
d\f(s)&=\big\{[\cA-\Si^H(\cS^H_1)^\top\cR^{-1}\cS^H_1-\cB \cR^{-1}\cS^H_1]\f\\
&\q~ -[\cL+\Si^H(\cS^H_3)^\top](\h\Si^H)^{-1}[\Si^H \cS^H_3\f-\b]-\si\big\}ds+\b dW(s),\q s\in[0,T], \\
 \f(T)&=-\xi,
\end{aligned}\right.
\ee
and SDE:
\bel{BLQ-fSDE}
\left\{
\begin{aligned}
d X(s)&=\big\{\big[(\cS^H_1)^\top\cR^{-1}\cB^\top+(\cS^H_1)^\top\cR^{-1}\cS^H_1\Si^H
+(\cS^H_3)^\top(\h\Si^H)^{-1}(\Si^H\cL^\top+\Si^H\cS^H_3\Si^H)\\
&\qq~-\cA^\top\big]X+(\cS^H_3)^\top(\h\Si^H)^{-1}(\Si^H\cS^H_3\f-\b)+(\cS^H_1)^\top\cR^{-1}\cS^H_1\f\big\}ds\\
&\q+\big\{\big[\cN^H(\h\Si^H)^{-1}(\Si^H\cL^\top+\Si^H\cS^H_3\Si^H)
-\cS^H_3\Si^H-\cL^\top\big]X\\
&\qq~+\cN^H(\h\Si^H)^{-1}[\Si^H\cS^H_3\f-\b]-\cS^H_3\f\big\}dW(s),\q s\in[0,T], \\
X(0)&=g.
\end{aligned}\right.
\ee

\begin{proposition}\label{Prop:backward}
Let {\rm\ref{ass:B3}} hold.
Then for any $\xi\in L^2_{\cF_T}(\Om;\dbR^n)$,
the unique optimal control of {\rm Problem (BLQ)} is given by
\begin{align}
\bar u_2&=\cR^{-1}\big\{[\cB^\top+\cS^H_1\Si^H]X+\cS^H_1\f\big\}\nn\\
&\q-\cR^{-1}\cS_2(\h\Si^H)^{-1}
\big\{\Si^H\cL^\top X+\Si^H\cS_3^H\Si^HX+\Si^H\cS_3\f-\b \big\},\label{bar-v}
\end{align}
where $\Si^H\in C(0,T;\dbS^n_-)$, $(\f,\b)\in  L^2_{\dbF}(\Om;C([0,T];\dbR^n))\times L^2_{\dbF}(0,T;\dbR^n)$
and $X\in L^2_{\dbF}(\Om;C([0,T];\dbR^n))$ are the unique solutions to  Riccati equation \rf{BLQ-Riccati-Equation},
BSDE \rf{BLQ-BSDE} and SDE \rf{BLQ-fSDE}, respectively.
Moreover, the value function $\bar U$ of {\rm Problem (BLQ)} is given explicitly by
\begin{align}
\bar U(\xi)&= -\lan \Si^H(0)g,\,g\ran-2\lan g,\f(0)\ran
-\dbE\lan H(T)\xi,\xi\ran+\dbE\int_0^T\Big\{\lan \cN^H(\h\Si^H)^{-1}\b,\,\b\ran
\nn\\
&\q +2\lan (S_3^H)^\top(\h\Si^H)^{-1}\b,\f\ran-\lan [(\cS^H_3)^\top(\h\Si^H)^{-1}\Si \cS^H_3
+(\cS^H_1)^\top\cR^{-1}\cS^H_1]\f,\,\f\ran\Big\}ds.\label{U-value}
\end{align}
\end{proposition}

By \autoref{Prop:backward}, we generalize the results obtained in  Sun, Wu, and Xiong  \cite{Sun-Wu-Xiong2021}
to the case with nonhomogeneous terms.
Since the proof of \autoref{Prop:backward} is similar to that of \cite[Theorem 6.3]{Sun-Wu-Xiong2021},
we omit it here. Even though, this extension will serve as a foundation for finding a Stackelberg equilibrium of Problem (SG)
(see \autoref{Thm:SOP-control}).
A key point in \cite{Sun-Wu-Xiong2021} is that by some transformation techniques,
the assumptions $\cQ\equiv 0$ and $\cG=0$ can be imposed without loss of generality.
However, the power of this approach is  very limited for our problem,
because as a trade off, the associated Riccati equation  depends additionally on an auxiliary function $H$
and the auxiliary function $H$ will cause some technical difficulties in exploring the connection
between Problems (SG) and (NG).
In Subsection \ref{sub:BLQ-t}, a new representation for the optimal control of Problem (BLQ) will be presented
and the auxiliary function $H$ will be removed.


\section{Stackelberg games}\label{Sec:SOG}
In this section, we shall  establish a general approach for finding the Stackelberg equilibrium of Problem (SG).
The  procedure will be divided into two steps.

\subsection{The follower's problem}
First, we are going to solve the follower's problem.
For any fixed control $u_2\in\cU_2$, the follower's problem (denoted by Problem (FLQ)) can be stated as follows:
Consider the
state equation
\bel{state-follower}\left\{\begin{aligned}
   dX(s) &=\big\{A(s)X(s) +B_1(s)u_1(s)+ B_2(s)u_2(s) \big\}ds\\
         &~\hp{=} +\big\{C(s)X(s)+D_1(s)u_1(s)+ D_2(s)u_2(s)\big\}dW(s),\q s\in[0,T], \\
    X(0) &= x,
\end{aligned}\right.\ee
and the cost functional
\begin{align}\label{cost-follower}
&\cJ_{u_2}(x;u_1)\equiv J(x;u_1,u_2)= \dbE\Big\{\lan GX(T),X(T)\ran
 +\int_0^T\[\lan QX,X\ran +\lan R_{1}u_1,u_1\ran+ \lan R_{2}u_2,u_2\ran\] ds \Big\}.
\end{align}
The follower (Player 1) wishes to find a control $\bar u_1\in\cU_1$, depending on $u_2$ and $x$, such that
\bel{}
\cJ_{u_2}(x;\bar u_1)=\inf_{u_1\in\cU_1}\cJ_{u_2}(x;u_1).
\ee
To find an optimal control of Problem (FLQ), we introduce the following assumption.

\begin{taggedassumption}{(H3)}\label{ass:H3}
There exists a constant $\l>0$ such that
$$
J(0;u_1,0)\ges \l\dbE\int_0^T|u_1(s)|^2ds,\q \forall u_1\in\cU_1.
$$
\end{taggedassumption}
Then by \cite[Corollary 4.7]{Sun-Li-Yong2016}, we have the following results.

\begin{proposition}\label{prop:follower}
Let {\rm\ref{ass:H1}--\ref{ass:H3}} hold.
Then for any  $u_2\in\cU_2$ and  $x\in\dbR^n$,
{\rm Problem (FLQ)} admits a unique optimal control $\bar u_1\in\cU_1$,
which admits the following closed-loop representation:
\bel{closed-loop-repre-X}
\bar u_1(s)=\bar\a_1(s;u_2,x)\equiv\Th(s) \bar X(s)+v(s)\equiv\Th(s) \bar X(s;x,u_2)+v(s;u_2),  \q s\in[0,T],
\ee
where %
      $$\begin{aligned}
        \Th &= -(R_{1}+D_1^\top P_1D_1)^{-1}(B_1^\top P_1+D_1^\top P_1C),\\
          v &= -(R_{1}+D_1^\top P_1D_1)^{-1}(B_1^\top Y+D_1^\top Z+D_1^\top P_1 D_2u_2),
      \end{aligned}$$
with $P_1\in C([0,T];\dbS^n)$ being the unique strongly regular solution of the  Riccati equation:
\bel{Ric-1}\left\{\begin{aligned}
  & \dot P_1+P_1A+A^\top P_1+C^\top P_1C+Q -(P_1B_1+C^\top P_1D_1)\\
  & \hp{\dot P}\q \times(R_{1}+D_1^\top P_1D_1)^{-1}(B_1^\top P_1+D_1^\top P_{1}C)=0, \\
  & P_1(T)=G,
\end{aligned}\right.\ee
$(Y,Z)\equiv(Y(\cd;u_2),Z(\cd;u_2))$ solving the BSDE:
\bel{eta-beta}\left\{\begin{aligned}
         dY(s) &=-\big[(A+B_1\Th)^\top Y +(C+D_1\Th)^\top Z +(C+D_1\Th)^\top P_1D_2u_2\\
                  &\hp{=-\big[} +P_1B_2u_2  \big]ds +Z dW(s), \q s\in[0,T],\\
          Y(T) &=0,
\end{aligned}\right.\ee
and $\bar X\equiv\bar X(\cd;x,u_2)$ satisfying the closed-loop system:
\bel{state-u2-star}\left\{\begin{aligned}
   d\bar X(s) &=\big[(A +B_1\Th)\bar X+B_1v+ B_2u_2 \big]ds\\
         &~\hp{=} +\big[(C +D_1\Th)\bar X+D_1v+ D_2u_2 \big]dW(s),\q s\in[0,T], \\
    \bar X(0) &= x.
\end{aligned}\right.\ee
\end{proposition}
\begin{remark}
Recall from \cite[Theorem 4.3]{Sun-Li-Yong2016} that the unique strongly regular solution $P_1$ of
Riccati equation \rf{Ric-1} satisfies
\bel{hat-R1}
R_{1}+D_1^\top P_1D_1\gg 0.
\ee
Since the optimal control $\bar u_1$ admits the closed-loop representation \rf{closed-loop-repre-X},
we have
\bel{}
J(x;\bar u_1(u_2,x),u_2)=J(x;\bar\a_1(u_2,x),u_2)=\inf_{u_1\in\cU_1} J(x;u_1,u_2),\q \forall u_2\in\cU_2,\, x\in\dbR^n.
\ee
\end{remark}

\subsection{The leader's problem and  Stackelberg equilibrium}
For any $x\in\dbR^n$ and $u_2\in\cU_2$, the follower's unique optimal control $\bar u_1$
can be given by \rf{closed-loop-repre-X}.
Knowing this, the leader's problem (denoted by Problem (LLQ)) becomes:
Find a control $\bar u_2\in\cU_2$ such that
\bel{}
J(x;\bar \a_1(\bar u_2,x),\bar u_2)=\sup_{u_2\in\cU_2}J(x;\bar \a_1( u_2,x), u_2).
\ee
From the facts that $\bar\a_1( u_2,x)=\Th \bar X(x,u_2)+v(u_2)$,
$\bar X(x,u_2)$ is the solution of \rf{state-u2-star}
and $v(u_2)$ is determined by BSDE \rf{eta-beta},
we see that Problem (LLQ) is an optimal control problem for  forward-backward SDEs.

\ms

By some straightforward calculations,  $J(x;\bar \a_1( u_2,x), u_2)$ can be rewritten as
\begin{align}
&J(x;\bar \a_1( u_2,x),u_2)=J(x;\Th \bar X+v,u_2)\nn\\
&\q= \lan P_1(0)x,x\ran+2\lan Y(0),x\ran+\dbE\Big\{\int_0^T\[\lan P_1D_2u_2,D_2u_2\ran+2\lan Y,B_2u_2\ran \nn\\
&\qq\, +2\lan Z,D_2u_2\ran -\big\lan(R_{1}+D_1^\top P_1D_1)^{-1}(B_1^\top Y+D_1^\top Z+D_1^\top P_1D_2u_2),\nn\\
&\qq\,\, (B_1^\top Y+D_1^\top Z+D_1^\top P_1D_2u_2)\big\ran+\lan R_{2} u_2, u_2\ran\] ds \Big\}.\label{cost-u2-star}
\end{align}
It shows that $J(x;\bar \a_1( u_2,x),u_2)$ is independent of the state process $\bar X$.
Noticing this key point, Problem (LLQ) is converted into an LQ optimal control problem for BSDEs
(with nonhomogeneous terms), which is a precondition for our subsequent analysis.

\ms

For simplicity, we denote
\bel{def-barA}\left\{\begin{aligned}
&\h R_1= D_1^\top P_1 D_1+R_1,\q \cA=(P_1B_1+C^\top P_1D_1)\h R_1^{-1}B_1^\top-A^\top,\\
&\cB=[(P_1B_1+C^\top P_1D_1)\h R_1^{-1}D_1^\top-C^\top]P_1 D_2-P_1B_2,\\
&\cC=(P_1B_1+C^\top P_1D_1)\h R_1^{-1}D_1^\top-C^\top,
\end{aligned}\right.\ee
and
\bel{def-barR}\left\{\begin{aligned}
  &\cR= D_2^\top P_1D_2-D_2^\top P_1D_1\h R_1^{-1}D_1^\top P_1 D_2+R_2,\\
  &\cQ=-B_1\h R_1^{-1}B_1^\top,\q \cN=-D_1\h R_1^{-1}D_1^\top,\q\cS_3= -D_1\h R_1^{-1}B_1^\top,\\
  &\cS_2= D_2^\top-D^\top_2P_1D_1\h R_1^{-1}D_1^\top,\q\cS_1= B_2^\top-D^\top_2P_1D_1\h R_1^{-1}B_1^\top.
\end{aligned}\right.\ee
With the above notations,  BSDE \rf{eta-beta} and  functional \rf{cost-u2-star} can be rewritten as
\bel{state-LLQ}\left\{\begin{aligned}
         dY(s) &=\big[\cA(s) Y(s) +\cB(s)u_2(s)+\cC(s)Z(s) \big]ds+Z(s) dW(s), \q s\in[0,T],\\
          Y(T) &=0,
\end{aligned}\right.\ee
and
\begin{align}\label{utility-LLQ}
J(x;\bar \a_1( u_2,x),u_2)
&=\lan P_1(0)x,x\ran+\dbE\Big\{\int_0^T\big[\lan \cR u_2, u_2\ran+\lan \cQ Y,Y\ran+\lan \cN Z,Z\ran +2\lan \cS_1 Y, u_2\ran\nn\\
&\q+2\lan \cS_2 Z,u_2\ran+2\lan \cS_3 Y,Z\ran\big] ds+2\lan Y(0),x\ran\Big\},
\end{align}
respectively.
Let $g=x$, $\cG=0$ and $\si\equiv 0$ in \rf{utility-backward} and \rf{backward-system}, respectively.
Comparing \rf{utility-LLQ} with \rf{utility-backward} yields that
\bel{cost-u2-U}
U(0;u_2)+\lan P_1(0)x,x\ran=J(x;\bar \a_1( u_2,x),u_2),
\ee
where $U$ is defined by \rf{utility-backward}.
Recall the definition \rf{def-NH} of $\cN^H$, $\cS_1^H$, $\cS_3^H$ and $\cL$.
By \autoref{Prop:backward}, we have the following result.

\begin{taggedassumption}{(H4)}\label{ass:H4}
There exists a constant $\l>0$ such that
$$
J(0;\bar \a_1( u_2,0),u_2)\les- \l\dbE\int_0^T|u_2(s)|^2ds,\q \forall u_2\in\cU_2.
$$
\end{taggedassumption}

\begin{proposition}\label{prop:leader}
Let {\rm\ref{ass:H1}}--{\rm\ref{ass:H4}} hold.
Then {\rm Problem (LLQ)} admits a unique optimal control:
\begin{align}
\bar u_2&=\cR^{-1}\big\{[\cB^\top+\cS^H_1\Si^H]-\cS_2(\h\Si^H)^{-1}
[\Si^H\cL^\top +\Si^H\cS_3^H\Si^H]\big\}\bar\cX,\label{bar-u2}
\end{align}
where $\Si^H$ is the unique solution of Riccati equation \rf{BLQ-Riccati-Equation}
and $\bar\cX$ is uniquely determined by the following SDE:
\bel{LLQ-fSDE}
\left\{
\begin{aligned}
d\bar\cX(s)&=\big\{(\cS^H_1)^\top\cR^{-1}\cB^\top+(\cS^H_1)^\top\cR^{-1}\cS^H_1\Si^H
+(\cS^H_3)^\top(\h\Si^H)^{-1}[\Si^H\cL^\top+\Si^H\cS^H_3\Si^H]\\
&\q-\cA^\top\big\}\bar\cX ds+\big\{\cN^H(\h\Si^H)^{-1}[\Si^H\cL^\top+\Si^H\cS^H_3\Si^H]
-\cS^H_3\Si^H-\cL^\top\big\}\bar\cX dW(s), \\
\bar\cX(0)&=x.
\end{aligned}\right.
\ee
\end{proposition}

\begin{proof}
Note from \rf{cost-u2-U} that
$$
U_0(0;u_2)=J(0;\bar \a_1( u_2,0),u_2),\q\forall u_2\in\cU_2,
$$
where $U_0$ is the utility functional $U$, defined by \rf{utility-backward}, with $\si=0$ and $g=0$.
Then assumption \ref{ass:H4} implies that \rf{uniform-concave-condition} holds.
Moreover, from the fact $\h R_1=D_1^\top P_1 D_1+R_1\gg 0$, we get that the coefficients $\cA,\cB,\cC$
and the weighting matrices $\cG, \cR, \cQ,\cN,\cS_1,\cS_2,\cS_3$ satisfy the condtion \rf{B1-1}.
Thus under \ref{ass:H1}--\ref{ass:H4}, the assumption \ref{ass:B3} holds.
Then by \autoref{Prop:backward}, Problem (LLQ) admits a unique optimal control
\begin{align}
\bar u_2&=\cR^{-1}\big\{[\cB^\top+\cS^H_1\Si^H]X+\cS^H_1\f\big\}\nn\\
&\q-\cR^{-1}\cS_2(\h\Si^H)^{-1}
\big\{\Si^H\cL^\top X+\Si^H\cS_3^H\Si^HX+\Si^H\cS_3\f-\b \big\},\label{LLQ-P1}
\end{align}
where $X$ and $(\f,\b)$ are the unique solutions of \rf{BLQ-fSDE} and \rf{BLQ-BSDE},
with $\si\equiv0$, $g=x$ and $\xi=0$, respectively.
Note that when $\si\equiv0$ and $\xi=0$,
the unique solution of BSDE \rf{BLQ-BSDE} is explicitly given by $(\f,\b)\equiv(0,0)$.
Using the facts $(\f,\b)\equiv(0,0)$ and $g=x$,
SDE \rf{BLQ-fSDE} can be rewritten as
\bel{}
\left\{
\begin{aligned}
dX(s)&=\big\{(\cS^H_1)^\top\cR^{-1}\cB^\top+(\cS^H_1)^\top\cR^{-1}\cS^H_1\Si^H
+(\cS^H_3)^\top(\h\Si^H)^{-1}[\Si^H\cL^\top+\Si^H\cS^H_3\Si^H]\\
&\q-\cA^\top\big\}X ds+\big\{\cN^H(\h\Si^H)^{-1}[\Si^H\cL^\top+\Si^H\cS^H_3\Si^H]
-\cS^H_3\Si^H-\cL^\top\big\}X dW(s), \\
X(0)&=x.
\end{aligned}\right.
\ee
It follows that $X=\bar\cX$. Substituting $(\f,\b)\equiv(0,0)$ and $X=\bar\cX$ into \rf{LLQ-P1},
we get \rf{bar-u2},
which completes the proof.
\end{proof}

We conclude this section with the following result.

\begin{theorem}\label{Thm:SOP-control}
Let {\rm\ref{ass:H1}}--{\rm\ref{ass:H4}} hold.
Then for any initial state $x\in\dbR^n$, the control pair $(\bar u_1,\bar u_2)\equiv(\bar\a_1(\bar u_2,x),\bar u_2)$,
obtained in {\rm\autoref{prop:follower}} and {\rm\autoref{prop:leader}},
is a Stackelberg equilibrium of {\rm Problem (SG)}.
\end{theorem}

\begin{proof}
From \autoref{prop:follower} and \autoref{prop:leader},
we see that the control pair $(\bar u_1,\bar u_2)=(\bar\a_1(\bar u_2,x),\bar u_2)$ satisfies
\begin{align}
J(x;\bar \a_1(u_2,x), u_2)&=\inf_{u_1\in\cU_1}J(x;u_1, u_2),\q \forall u_2\in\cU_2,\q\\
J(x;\bar \a_1(\bar u_2,x),\bar u_2)&=\sup_{u_2\in\cU_2}J(x;\bar \a_1(u_2,x), u_2).
\end{align}
It follows from \autoref{def:S-junheng} and \autoref{def:S-junheng*} that $(\bar u_1,\bar u_2)=(\bar\a_1(\bar u_2,x),\bar u_2)$
is a Stackelberg equilibrium of Problem (SG).
\end{proof}

\begin{remark}
By \cite[Theorem 3.1]{Sun-Wu-Xiong2021}, the following condition is necessary for open-loop solvability
 of Problem (LLQ):
\bel{necessary-LLQ}
U_0(0;u_2)=J(0;\bar \a_1( u_2,0),u_2)\les 0,\q \forall u_2\in\cU_2.
\ee
Then assumption \ref{ass:H4} is almost necessary for the existence of
an optimal control of Problem (LLQ).
When Problem (SG) only satisfies \ref{ass:H1}--\ref{ass:H3} and \rf{necessary-LLQ},
one can apply the  perturbation approach,
developed in \cite{Sun-Li-Yong2016,Wang-Sun-Yong2019,Sun-Wang-Wu2021},
to find the Stackelberg equilibrium (if exists).
\end{remark}

\section{Further analysis of the Stackelberg games}\label{sec:Further-Analysis}
In \autoref{Thm:SOP-control}, it has been shown that under  {\rm\ref{ass:H1}}--{\rm\ref{ass:H4}},
Problem (SG) admits a Stackelberg equilibrium $(\bar u_1,\bar u_2)=(\bar\a_1(\bar u_2,x),\bar u_2)$.
However, the assumption \ref{ass:H4} is usually difficult to verify,
because it is involved with the optimal strategy $\bar \a_1(\cd,\cd)$ of the follower.
In this section, we shall provide a new condition,
independent of $\bar \a_1(\cd,\cd)$, to ensure that \ref{ass:H4} holds.
Furthermore, under this condition, a closed-loop representation for the  Stackelberg equilibrium
of Problem (SG) is obtained by a closer investigation of backward stochastic LQ optimal control problems.

\subsection{Uniform concavity of the functional}
For any $t\in[0,T)$, we first introduce the following game problem over $[t,T]$:
Consider the state equation
\bel{statet}\left\{\begin{aligned}
   dX(s) &=\big\{A(s)X(s) +B_1(s)u_1(s) +B_2(s)u_2(s) \big\}ds\\
         &~\hp{=}+\big\{C(s)X(s) +D_1(s)u_1(s)+ D_2(s)u_2(s)\big\}dW(s),\q s\in[t,T], \\
    X(t) &= x,
\end{aligned}\right.\ee
and the criterion  functional
\bel{costt}
J(t,x;u_1,u_2)= \dbE\Big\{\lan GX(T),X(T)\ran
+\int_t^T\big[\lan QX,X\ran+\lan R_1 u_1,u_1\ran+\lan R_2u_2,u_2\ran\big]ds  \Big\},
\ee
where $u_i\in\cU_i[t,T]\equiv L^2_{\dbF}([t,T];\dbR^{m_i});i=1,2$. Then,
\bel{}
J(0,x;u_1,u_2)=J(x;u_1,u_2),\q \forall x\in\dbR^n,\, u_i\in\cU_i;\,i=1,2,
\ee
where $J(x;u_1,u_2)$, defined by \rf{cost}, is the  criterion  functional of Problem (SG).
The following result shows that under \ref{ass:H3},
the mapping $u_1\mapsto J(t,0;u_1,0)$ is uniformly convex for any $t\in[0,T)$.

\begin{lemma}\label{lem:convex}
Let {\rm\ref{ass:H1}--\ref{ass:H3}} hold. Then for any $t\in[0,T)$,
\bel{lem:convex-main}
J(t,0;u_1,0)\ges \l\dbE\int_t^T|u_1(s)|^2ds,\q \forall u_1\in\cU_1[t,T],
\ee
where $\l>0$ is the same as that in {\rm\ref{ass:H3}}.
\end{lemma}

\begin{proof}
For any $t\in[0,T)$ and $u_1\in\cU_1[t,T]$, define
\bel{}
[u_1\oplus_t 0](s)\deq\left\{
\begin{aligned}
&u_1(s),\q s\in[t,T],\\
&\q 0,\q\q s\in[0,t).
\end{aligned}\right.
\ee
It is clearly seen that $u_1\oplus_t 0\in\cU_1[0,T]\equiv\cU_1$ and
\bel{J(ut)}
J(t,0;u_1,0)=J(0,0;u_1\oplus_t 0,0)=J(0;u_1\oplus_t 0,0).
\ee
From \ref{ass:H3}, we have
\bel{J(ut1)}
J(0;u_1\oplus_t 0,0)\ges \l\dbE\int_0^T|[u_1\oplus_t 0](s)|^2ds=\l\dbE\int_t^T|u_1(s)|^2ds.
\ee
Combining \rf{J(ut)} and \rf{J(ut1)} together, we get \rf{lem:convex-main} immediately.
\end{proof}

For any $(t,x)\in[0,T)\times\dbR^n$,
by \autoref{lem:convex} and \autoref{prop:follower} (with the initial time $0$ replaced by $t$), we have
\begin{align}
&J(t,x;\bar \a_1( u_2,t,x),u_2)\les J(t,x;u_1,u_2),\q \forall u_1\in\cU_1[t,T],\, u_2\in\cU_2[t,T],\label{J(t)-inquality}
\end{align}
where $\bar \a_1( u_2,t,x)\equiv\bar \a_1(\cd; u_2,t,x)$ is defined by \rf{closed-loop-repre-X} with
the initial time of \rf{state-u2-star} replaced by $t$.
Moreover, similar to \rf{utility-LLQ}, we have
\begin{align}\label{utility-LLQ-t}
&J(t,x;\bar \a_1( u_2,t,x),u_2)
= \dbE\Big\{\int_t^T\big[\lan \cR u_2, u_2\ran+\lan \cQ Y,Y\ran+\lan \cN Z,Z\ran\nn\\
&\q+2\lan \cS_1 Y, u_2\ran+2\lan \cS_2 Z,u_2\ran+2\lan \cS_3 Y,Z\ran\big] ds+2\lan Y(t),x\ran\Big\}
+\lan P_1(t)x,x\ran,
\end{align}
where $P_1$ is the unique solution to Riccati equation \rf{Ric-1}, $(Y,Z)$ is uniquely determined by
\bel{state-LLQ-t}\left\{\begin{aligned}
         dY(s) &=\big[\cA(s) Y(s) +\cB(s)u_2(s)+\cC(s)Z(s) \big]ds+Z(s) dW(s), \q s\in[t,T],\\
          Y(T) &=0,
\end{aligned}\right.\ee
and the coefficients are defined by \rf{def-barA}--\rf{def-barR}.
The optimal control problem with state equation \rf{state-LLQ-t} and utility \rf{utility-LLQ-t}
is a backward LQ problem over the time horizon $[t,T]$.
Next, we show that the following condition is sufficient for the uniform concavity of
the mapping $u_2\mapsto J(t,0;\bar \a_1( u_2,t,0),u_2)$.

\begin{taggedassumption}{(H5)}\label{ass:H5}
There exists a constant $\l>0$ such that
$$
J(0;0,u_2)\les- \l\dbE\int_0^T|u_2(s)|^2ds,\q \forall u_2\in\cU_2,
$$
where $J$ is defined by \rf{cost}.
\end{taggedassumption}

\begin{proposition}\label{Prop:saddle-point-control}
Let {\rm\ref{ass:H1}}--{\rm\ref{ass:H3}} and {\rm\ref{ass:H5}} hold.
Then
\bel{uniform-concave-U}
J(t,0;\bar \a_1( u_2,t,0),u_2)\les- \l\dbE\int_t^T|u_2(s)|^2ds,\q \forall u_2\in\cU_2[t,T],\, t\in[0,T),
\ee
where $J(t,0;\bar \a_1( u_2,t,0),u_2)$ is defined by \rf{utility-LLQ-t}.
In particular, assumption {\rm\ref{ass:H5}} implies that {\rm\ref{ass:H4}} holds.
\end{proposition}

\begin{proof}
Recall from \rf{J(t)-inquality} that for any $(t,x)\in[0,T)\times\dbR^n$ and $u_2\in\cU_2[t,T]$,
we have
\bel{}
J(t,x;\bar \a_1(u_2,t,x), u_2)\les J(t,x;u_1, u_2),\q \forall u_1\in\cU_1[t,T].
\ee
In particular, taking $x=0$ and $u_1=0$, the above implies
\bel{uniform-concave-U1}
J(t,0;\bar \a_1(u_2,t,0), u_2)\les J(t,0;0, u_2),\q \forall u_2\in\cU_2[t,T].
\ee
Moreover, by the similar arguments to those employed in \autoref{lem:convex},  we get
\bel{uniform-concave-U2}
J(t,0;0,u_2)\les- \l\dbE\int_t^T|u_2(s)|^2ds,\q \forall u_2\in\cU_2[t,T].
\ee
Combining \rf{uniform-concave-U1} and \rf{uniform-concave-U2} together,
we obtain \rf{uniform-concave-U} immediately.
\end{proof}

The following examples are devoted to comparing the assumptions \ref{ass:H4} and \ref{ass:H5}.

\begin{example}
For any $x\in\dbR$, consider the  one-dimensional state equation
\bel{state-example1}\left\{\begin{aligned}
   \dot{X}(s) &=u_2(s),\q s\in[0,1],\\
     X(0)&= x,
\end{aligned}\right.\ee
and the quadratic functional
\bel{cost-example1}
J(x;u_1,u_2)=\int_0^1\big[|u_1(s)|^2-|u_2(s)|^2]ds.
\ee
It is directly checked that
$$
\bar \a_1(s;u_2,x)=0,\q s\in[0,1].
$$
Then
\bel{}
J(0; \bar \a_1(u_2,0),u_2)=J(0;0,u_2),\q \forall u_2\in\cU_2.
\ee
Thus, in the example, the assumptions \ref{ass:H4} and \ref{ass:H5} are equivalent.
\end{example}

\begin{example}\label{example2}
For any initial pair $(t,x)\in[0,4)\times\dbR$, consider the  one-dimensional state equation
\bel{state-example2}\left\{\begin{aligned}
   \dot{X}(s) &=u_1(s)-u_2(s),\q s\in[t,4],\\
     X(t)&= x,
\end{aligned}\right.\ee
and the quadratic functional
\bel{cost-example2}
J(t,x;u_1,u_2)=\int_t^4\big[|X(s)|^2+|u_1(s)|^2-2|u_2(s)|^2\big]ds.
\ee
It is direct to see that
$$
J(0,0;u_1,0)\ges \int_0^4|u_1(s)|^2ds,\q \forall u_1\in\cU_1[0,4],
$$
which implies that  \ref{ass:H3} holds.
By \autoref{prop:follower} and \autoref{lem:convex},
we know that for any initial pair $(t,x)\in[0,4)\times\dbR$ and $u_2\in\cU_2[t,4]$,
the follower (Player 1) admits a unique optimal control $\bar u_1\equiv\bar\a_1(u_2,t,x)$.
Note that
$$
J(0,0;0,\l)=\int_0^4|\l s|^2-2\l^2 ds={40\over 3}\l^2\to\i,\q \hbox{as}\q\l\to\i.
$$
Then the following condition does not hold:
\bel{con-example}
J(0,0;0,u_2)\les 0,\q \forall u_2\in\cU_2[0,4],
\ee
due to which the example does not satisfy assumption \ref{ass:H5}.
Even so, we still have
\bel{concave-example2}
J(t,0;\bar\a_1(u_2,t,0),u_2)\les J(t,0;u_2,u_2)=-\int_t^4|u_2(s)|^2ds,\q \forall u_2\in\cU_2[t,4],\,t\in[0,4),
\ee
which implies that \ref{ass:H4} still holds. It then follows from \autoref{Thm:SOP-control} that
the game  admits a Stackelberg equilibrium at any initial pair $(t,x)\in[0,4)\times\dbR$.
We point out that condition \rf{con-example} is necessary for the existence of an open-loop saddle point (see  \cite[Theorem 3.3]{Sun2020}).
Since the criterion functional \rf{cost-example2} does not satisfy \rf{con-example},
the game does not have an open-loop saddle point.

\end{example}

\begin{remark}
The combination of \ref{ass:H3} and \ref{ass:H5} is referred to as
a {\it uniform convexity-concavity condition} ((UCC) condition, for short) by Sun \cite{Sun2020}.
In \autoref{example2}, it has been shown that assumption \ref{ass:H4} is strictly weaker than \ref{ass:H5},
due to which we would like to call the  assumptions \ref{ass:H3}--\ref{ass:H4}
a {\it weak uniform convexity-concavity condition}.
\end{remark}

\subsection{Further results of backward  stochastic LQ optimal control problems}\label{sub:BLQ-t}
For any $(t,x)\in[0,T)\times\dbR^n$, we begin with this subsection by introducing the
following backward stochastic LQ optimal control problem over $[t,T]$
(denoted by Problem (BLQ$_{[t,T]}$)): Consider the state equation
\bel{state-LLQ-Ut}\left\{\begin{aligned}
         dY(s) &=\big[\cA(s) Y(s) +\cB(s)u_2(s)+\cC(s)Z(s) \big]ds+Z(s) dW(s), \q s\in[t,T],\\
          Y(T) &=\xi,
\end{aligned}\right.\ee
and the utility functional
\begin{align}\label{utility-LLQ-Ut}
&U(t,\xi;u_2)
= \dbE\Big\{\int_t^T\big[\lan \cR u_2, u_2\ran+\lan \cQ Y,Y\ran+\lan \cN Z,Z\ran
+2\lan \cS_1 Y, u_2\ran \nn\\
&\qq\qq\qq +2\lan \cS_2 Z,u_2\ran+2\lan \cS_3 Y,Z\ran\big] ds+2\lan Y(t),x\ran\Big\},
\end{align}
where the coefficients are defined by \rf{def-barA}--\rf{def-barR}.
Denote  the utility functional $U(t,\xi;u_2)$ with  $x=0$ by $U_0(t,\xi;u_2)$.
The following results show that Problem (BLQ$_{[t,T]}$) can be solved by introducing a new Riccati equation.

\begin{proposition}\label{pro:Riccati-BLQ}
Let {\rm\ref{ass:H1}--\ref{ass:H3}} and {\rm\ref{ass:H5}} hold.
Then the following Riccati equation admits a unique
negative semidefinite solution $\Si\in C(0,T;\dbS^n_-)$:
\bel{BLQ-Riccati-Equation1}
\left\{\begin{aligned}
&\dot{\Si}-\Si \cA^\top-\cA\Si+\Si\cS_1^\top\cR^{-1}\cB^\top+\cB\cR^{-1}\cS_1\Si
+\Si\cS^\top_1\cR^{-1}\cS_1\Si-\Si \cQ\Si+\cB\cR^{-1}\cB^\top\\
&\q +[\cC-\cB\cR^{-1}\cS_2-\Si\cS_1^\top\cR^{-1}\cS_2 +\Si\cS_3^\top]
[I+\Si\cN-\Si\cS_2^\top\cR^{-1}\cS_2]^{-1}\\
&\q \times\Si[\cC^\top-\cS_2^\top\cR^{-1}\cB^\top-\cS_2^\top\cR^{-1}\cS_1\Si +\cS_3\Si]=0,
\qq t\in[0,T],\\
& \Si(T)=0.
\end{aligned}\right.
\ee
Moreover, the function $\h\Si$, defined by
\bel{hat-Si1}
\h\Si\equiv I+\Si\cN-\Si\cS_2^\top\cR^{-1}\cS_2,
\ee
is invertible with $\h\Si^{-1}\in L^\i(0,T;\dbR^{n\times n})$ and $\h\Si^{-1}\Si\in L^\i(0,T;\dbS^{n})$.
\end{proposition}

\begin{proof}
Comparing \rf{utility-LLQ-Ut} with \rf{utility-LLQ-t} yields
\bel{}
U_0(t,0;u_2)=J(t,0;\bar \a_1( u_2,t,0),u_2),\q\forall u_2\in\cU_2[t,T],
\ee
where $J(t,x;u_1,u_2)$ is defined by \rf{costt}.
Under \ref{ass:H5}, by \autoref{Prop:saddle-point-control} we have
\bel{U0-convex}
U_0(t,0;u_2)=J(t,0;\bar \a_1( u_2,t,0),u_2)\les -\l\dbE\int_t^T|u_2(s)|^2ds,\q\forall u_2\in\cU_2[t,T],\,t\in[0,T).
\ee
Denote
\bel{def-cL}
\ti\cC= \cC-\cB\cR^{-1}\cS_2, \q \ti\cS_3= \cS_3-\cS_2^\top\cR^{-1}\cS_1,\q
\ti\cN=\cN-\cS_2^\top\cR^{-1}\cS_2,\q v_2=u_2+\cR^{-1}\cS_2Z.
\ee
Then state equation \rf{state-LLQ-Ut} and utility functional \rf{utility-LLQ-Ut} can be rewritten as:
\bel{state-LLQ-Ut1}\left\{\begin{aligned}
         dY(s) &=\big[\cA(s) Y(s) +\cB(s)v_2(s)+\ti\cC(s)Z(s) \big]ds+Z(s) dW(s), \q s\in[t,T],\\
          Y(T) &=\xi,
\end{aligned}\right.\ee
and
\begin{align}\label{utility-LLQ-Ut1}
&\ti U(t,\xi;v_2)\deq\dbE\Big\{\int_t^T\big[\lan \cR v_2, v_2\ran+\lan \cQ Y,Y\ran+\lan\ti\cN Z,Z\ran
\nn\\
&\qq +2\lan \cS_1 Y, v_2\ran +2\lan\ti\cS_3 Y,Z\ran\big] ds+2\lan Y(t),x\ran\Big\}= U(t,\xi;u_2).
\end{align}
Similarly, we denote  the utility functional $\ti U(t,\xi;v_2)$ with  $x=0$ by $\ti U_0(t,\xi;v_2)$.
By the standard results of BSDEs, we get
\bel{v2}
\dbE\int_t^T|v_2(s)|^2\les K\dbE\int_t^T\big[|u_2(s)|^2+|Z(s)|^2\big]ds\les K\dbE\[\int_t^T|u_2(s)|^2ds+|\xi|^2\].
\ee
Here, $K>0$  stands for a generic constant which could be different from line to line and is  independent of $t$.
Then by \rf{U0-convex} and \rf{v2} (with $\xi=0$), we get
\begin{align}
\ti U_0(t,0;v_2)&= U_0(t,0;u_2)\les -\l\dbE\int_t^T|u_2(s)|^2ds\nn\\
&\les -{\l\over K}\dbE\int_t^T|v_2(s)|^2ds,\q\forall v_2\in\cU_2[t,T],\,t\in[0,T).
\end{align}
Thus by \cite[Theorem 5.1 and Corollary 5.3]{Sun-Wu-Xiong2021}, there exists a constant $ k_0>0$ such that for any $k>k_0$
the following Riccati equation
\bel{}
\left\{\begin{aligned}
&\dot{P}_k+P_k\cA+\cA^\top P_k-\begin{pmatrix} \ti\cC^\top P_k+\ti\cS_3\\
                                      \cB^\top P_k+\cS_1\end{pmatrix}^\top
                                      \begin{pmatrix} \ti\cN+P_k,&0\\
                                      0& \cR\end{pmatrix}^\top
                                      \begin{pmatrix} \ti\cC^\top P_k+\ti\cS_3\\
                                      \cB^\top P_k+\cS_1\end{pmatrix}=0,\q t\in[0,T],\\
&P_k(T)=-kI,
\end{aligned}\right.
\ee
admits a unique solution $P_k\in C([0,T];\dbS_-^n)$.
Then applying the arguments employed in the proof of \cite[Theorem 6.2]{Sun-Wu-Xiong2021},
we get that $\Si\equiv\lim_{k\to\i} P_k^{-1}$ is the unique solution to the following Riccati equation
\bel{}
\left\{\begin{aligned}
&\dot{\Si}-\Si \cA^\top-\cA\Si+\Si\cS_1^\top\cR^{-1}\cB^\top+\cB\cR^{-1}\cS_1\Si
+\Si\cS^\top_1\cR^{-1}\cS_1\Si-\Si \cQ\Si+\cB\cR^{-1}\cB^\top\\
&\q +[\ti\cC +\Si\ti\cS_3^\top]
[I+\Si\ti\cN]^{-1}\Si[\ti\cC^\top +\ti\cS_3\Si]=0,
\qq t\in[0,T],\\
& \Si(T)=0.
\end{aligned}\right.
\ee
Moreover,
$I+\Si\ti\cN$
is invertible with $[I+\Si\ti\cN]^{-1}\in L^\i(0,T;\dbR^{n\times n})$ and
$[I+\Si\ti\cN]^{-1}\Si\in L^\i(0,T;\dbS^{n})$.
Then by the definition \rf{def-cL} of $\ti\cC$, $\ti\cS_3$ and $\ti\cN$,
we get the desired results immediately.
\end{proof}

Compared with \rf{BLQ-Riccati-Equation}, Ricaati equation \rf{BLQ-Riccati-Equation1}
does not depend on the auxiliary function $H$.
This new feature will play a crucial role in the proof of \autoref{Thm:Riccati}.
A challenging problem is to establish the well-posedness of Riccati equation  \rf{BLQ-Riccati-Equation1}
under an assumption like \ref{ass:H4}. We hope to come back in our future publications.
With the unique solution $\Si$ of \rf{BLQ-Riccati-Equation}, we introduce the following BSDE:
\bel{BLQ-BSDE1}
\left\{\begin{aligned}
d\f(s)&=\big\{(\cA+\Si\cQ-\Si\cS_1^\top\cR^{-1}\cS_1-\cB\cR^{-1}\cS_1)\f+[\cB\cR^{-1}\cS_2+\Si\cS_1^\top\cR^{-1}\cS_2\\
&\q~ -\cC-\Si\cS_3^\top]\h\Si^{-1}[(\Si \cS_3-\Si\cS_2^\top \cR^{-1}\cS_1)\f-\b]\big\}ds+\b dW(s),\q s\in[0,T], \\
 \f(T)&=-\xi,
\end{aligned}\right.
\ee
and SDE:
\bel{BLQ-fSDE1}
\left\{
\begin{aligned}
d X(s)&=\big\{-\big[\cA^\top+\cQ\Si-\cS_1^\top\cR^{-1}\cB^\top-\cS_1^\top\cR^{-1}\cS_1\Si
-(\cS_3^\top-\cS_1^\top\cR^{-1}\cS_2)\\
&\qq\times\h\Si^{-1}(\Si\cC^\top-\Si\cS_2^\top\cR^{-1}\cB^\top-\Si\cS_2^\top\cR^{-1}\cS_1\Si+\Si\cS_3\Si)\big]X\\
&\qq+(\cS_3^\top-\cS_1^\top\cR^{-1}\cS_2)\h\Si^{-1}(\Si\cS_3\f-\Si\cS_2^\top\cR^{-1}\cS_1\f-\b)
-\cQ\f+\cS_1^\top\cR^{-1}\cS_1\f\big\}ds\\
&\q+\big\{-\big[\cC^\top-\cS_2^\top\cR^{-1}\cB^\top
+(\cS_3-\cS_2^\top\cR^{-1}\cS_1)\Si-(\cN-\cS_2^\top\cR^{-1}\cS_2)\\
&\qq~\times\h\Si^{-1}(\Si\cC^\top-\Si\cS_2^\top\cR^{-1}\cB^\top-\Si\cS_2^\top\cR^{-1}\cS_1\Si+\Si\cS_3\Si)\big]X\\
&\qq~+(\cN-\cS_2^\top\cR^{-1}\cS_2)\h\Si^{-1}[(\Si\cS_3-\Si\cS_2^\top\cR^{-1}\cS_1)\f-\b]\\
&\qq~-(\cS_3-\cS_2^\top \cR^{-1}\cS_1)\f\big\}dW(s),\q s\in[0,T], \\
X(0)&=x.
\end{aligned}\right.
\ee

Then by the standard argument employed in backward stochastic LQ  optimal control problems
(see \cite{Sun-Wu-Xiong2021}, for example),
we can obtain the unique optimal control of Problem (BLQ$_{[0,T]}$).
The uniqueness of optimal controls of Problem (BLQ$_{[0,T]}$) comes from the uniform
concavity of the utility functional (see \autoref{Prop:saddle-point-control}).

\begin{proposition}\label{Prop:backward1}
Let {\rm\ref{ass:H1}--\ref{ass:H3}} and {\rm\ref{ass:H5}} hold.
Then for any $\xi\in L^2_{\cF_T}(\Om;\dbR^n)$,
the unique optimal control of {\rm Problem (BLQ)$_{[0,T]}$} is given by
\begin{align}
\bar u_2&=\cR^{-1}[\cB^\top+\cS_1\Si-\cS_2\h\Si^{-1}(\Si\cC^\top-\Si\cS_2^\top\cR^{-1}\cB^\top
-\Si\cS_2^\top\cR^{-1}\cS_1\Si+\Si\cS_3\Si)]X\nn\\
&\q+\cR^{-1}\cS_1\f-\cR^{-1}\cS_2\h\Si^{-1}[(\Si\cS_3-\Si\cS_2^\top\cR^{-1}\cS_1)\f-\b],\label{bar-v1}
\end{align}
where $\Si\in C(0,T;\dbS^n_-)$, $(\f,\b)\in  L^2_{\dbF}(\Om;C([0,T];\dbR^n))\times L^2_{\dbF}(0,T;\dbR^n)$
and $X\in L^2_{\dbF}(\Om;C([0,T];\dbR^n))$ are the unique solutions of  Riccati equation \rf{BLQ-Riccati-Equation1},
BSDE \rf{BLQ-BSDE1} and SDE \rf{BLQ-fSDE1}, respectively.
Moreover, the value function $U$ of {\rm Problem (BLQ)$_{[0,T]}$} can be represented as
\begin{align}
&U(0,\xi;\bar u_2)= -\lan\Si(0) x,x\ran-2\lan x,\f(0)\ran+\dbE\int_0^T
\Big\{\big\lan [\cQ-(\cS_3^\top-\cS_1^\top \cR^{-1}\cS_2)\h\Si^{-1}\Si (\cS_3-\cS_2^\top \cR^{-1}\cS_1)\nn\\
&\q -\cS_1^\top\cR^{-1}\cS_1]\f,\,\f\big\ran+2\big\lan(\cS_3^\top-\cS_1^\top\cR^{-1}\cS_2 )\h\Si^{-1}\b,\,\a\big\ran
+\big\lan(\cN-\cS_2^\top\cR^{-1}\cS_2)\h\Si^{-1}\b,\,\b\big\ran \Big\}ds.\label{U-value1}
\end{align}
\end{proposition}

By \autoref{Prop:backward1}, we can rewrite \autoref{prop:leader} as follows.

\begin{proposition}\label{prop:leader1}
Let {\rm\ref{ass:H1}--\ref{ass:H3}} and {\rm\ref{ass:H5}} hold.
Then {\rm Problem (LLQ)} admits a unique optimal control:
\begin{align}
\bar u_2&=\cR^{-1}[\cB^\top+\cS_1\Si-\cS_2\h\Si^{-1}(\Si\cC^\top-\Si\cS_2^\top\cR^{-1}\cB^\top
-\Si\cS_2^\top\cR^{-1}\cS_1\Si+\Si\cS_3\Si)]\bar\cX,\label{bar-u21}
\end{align}
where $\Si$ is the unique solution to Riccati equation \rf{BLQ-Riccati-Equation1}
and $\bar\cX$ is uniquely determined by the following SDE:
\bel{LLQ-fSDE1}
\left\{
\begin{aligned}
d\bar\cX(s)&=-\big[\cA^\top+\cQ\Si-\cS_1^\top\cR^{-1}\cB^\top-\cS_1^\top\cR^{-1}\cS_1\Si
-(\cS_3^\top-\cS_1^\top\cR^{-1}\cS_2)\\
&\qq\times\h\Si^{-1}(\Si\cC^\top-\Si\cS_2^\top\cR^{-1}\cB^\top-\Si\cS_2^\top\cR^{-1}\cS_1\Si+\Si\cS_3\Si)\big]\bar\cX ds\\
&\q-\big[\cC^\top-\cS_2^\top\cR^{-1}\cB^\top
+(\cS_3-\cS_2^\top\cR^{-1}\cS_1)\Si-(\cN-\cS_2^\top\cR^{-1}\cS_2)\\
&\qq~\times\h\Si^{-1}(\Si\cC^\top-\Si\cS_2^\top\cR^{-1}\cB^\top-\Si\cS_2^\top\cR^{-1}\cS_1\Si
+\Si\cS_3\Si)\big]\bar\cX dW(s),\q s\in[0,T], \\
\bar\cX(0)&=x.
\end{aligned}\right.
\ee
\end{proposition}

\subsection{Closed-loop representation for the Stackelberg equilibrium}\label{subsec:CR}
In this subsection, we shall show that the Stackelberg equilibrium $(\bar u_1,\bar u_2)$
obtained in \autoref{Thm:SOP-control} admits a closed-loop representation.

\begin{theorem}\label{Thm:SOP-control1}
Let {\rm\ref{ass:H1}--\ref{ass:H3}} and {\rm\ref{ass:H5}} hold.
Let $P_1\in C([0,T];\dbS^n)$ and $\Si\in C([0,T];\dbS_-^n)$ be the unique solutions to Riccati equations
\rf{Ric-1} and \rf{BLQ-Riccati-Equation1}, respectively.
Then {\rm Problem (SG)} has a Stackelberg equilibrium $(\h u_1,\h u_2)\in\cU_1\times\cU_2$,
which admits the following closed-loop representation:
\begin{align}
\h u_1=\h \Th_1\h X&\equiv \h R_1^{-1}\big\{B_1^\top\Si-B_1^\top P_1-D_1^\top P_1C
-D_1^\top\h\Si^{-1}
(\Si\cC^\top-\Si\cS_2^\top\cR^{-1}\cB^\top\nn\\
&\q-\Si\cS_2^\top\cR^{-1}\cS_1\Si+\Si\cS_3\Si)
-D_1^\top P_1 D_2\cR^{-1}[\cB^\top+\cS_1\Si-\cS_2\h\Si^{-1}(\Si\cC^\top\nn\\
&\q-\Si\cS_2^\top\cR^{-1}\cB^\top
-\Si\cS_2^\top\cR^{-1}\cS_1\Si+\Si\cS_3\Si)]\big\}\h X,\label{thm:closed-loop-system-SO1}\\
\h u_2=\h\Th_2\h X&\equiv \cR^{-1}\big[\cB^\top+\cS_1\Si-\cS_2\h\Si^{-1}(\Si\cC^\top-\Si\cS_2^\top\cR^{-1}\cB^\top
-\Si\cS_2^\top\cR^{-1}\cS_1\Si+\Si\cS_3\Si)\big]\h X,
\label{thm:closed-loop-system-SO2}
\end{align}
with $\h X$ being the unique solution of the closed-loop system:
\bel{state-closed-loop-SO}\left\{\begin{aligned}
   d\h X(s) &=\big\{A(s)\h X(s) +B_1(s)\h\Th_1 (s)\h X(s) +B_2(s)\h\Th_2(s)\h X(s) \big\}ds\\
         &~\hp{=}+\big\{C(s)\h X(s) +D_1(s)\h\Th_1(s)\h X(s)+ D_2(s)\h\Th_2(s)\h X(s)\big\}dW(s),\q s\in[0,T], \\
    \h X(0) &= x.
\end{aligned}\right.\ee
Moreover,
\begin{align}
J(x;\h u_1,\h u_2)=\lan(P_1(0)-\Si(0))x,x\ran, \q \forall x\in\dbR^n.\label{SOP-V}
\end{align}
\end{theorem}

\begin{proof}
Taking $u_2=\bar u_2$ in \rf{eta-beta}, then the unique solution $(Y,Z)$ of \rf{eta-beta} can be given by
\begin{align}
Y&=-\Si \bar\cX,\q
Z=\h\Si^{-1}(\Si\cC^\top-\Si\cS_2^\top\cR^{-1}\cB^\top
-\Si\cS_2^\top\cR^{-1}\cS_1\Si+\Si\cS_3\Si)\bar\cX,
\end{align}
where $\bar u_2$ and $\bar\cX$ are determined by \rf{bar-u21} and \rf{LLQ-fSDE1}, respectively.
Substituting the above into \rf{closed-loop-repre-X} yields that
\begin{align}
\bar u_1&=\bar\a_1(\bar u_2,x)= -\h R_1^{-1}(B_1^\top P_1+D_1^\top P_1C)\bar X
+\h R_1^{-1}B_1^\top\Si\bar\cX-\h R_1^{-1}D_1^\top\h\Si^{-1}
\big(\Si\cC^\top-\Si\cS_2^\top\cR^{-1}\cB^\top\nn\\
&\q-\Si\cS_2^\top\cR^{-1}\cS_1\Si+\Si\cS_3\Si\big)\bar\cX
-\h R_1^{-1}D_1^\top P_1 D_2\cR^{-1}\big[\cB^\top+\cS_1\Si-\cS_2\h\Si^{-1}(\Si\cC^\top\nn\\
&\q-\Si\cS_2^\top\cR^{-1}\cB^\top
-\Si\cS_2^\top\cR^{-1}\cS_1\Si+\Si\cS_3\Si)\big]\bar\cX,\label{bar-u1}
\end{align}
where $\bar X$ is the unique solution of \rf{state-u2-star} with $u_2=\bar u_2$; that is
\bel{state-u2-star11}\left\{\begin{aligned}
d\bar X(s)&=\big\{A\bar X +B_1\h R_1^{-1}(B_1^\top P_1+D_1^\top P_1C)[\bar\cX-\bar X]+B_1\h\Th_1 \bar\cX + B_2\h\Th_2 \bar\cX  \big\}ds\\
&\hp{ =\,} +\big\{C\bar X +D_1\h R_1^{-1}(B_1^\top P_1+D_1^\top P_1C)[\bar\cX-\bar X]+D_1\h\Th_1\bar\cX + D_2\h\Th_2\bar\cX  \big\}dW(s), \\
\bar X(0)& = x.
\end{aligned}\right.\ee
To prove that $(\bar u_1,\bar u_2)=(\h u_1,\h u_2)$,
by comparing \rf{bar-u1} and \rf{bar-u21} with \rf{thm:closed-loop-system-SO1} and \rf{thm:closed-loop-system-SO2},
it suffices to show that  $\bar X=\bar\cX=\h X$.
If equation \rf{LLQ-fSDE1} can be rewritten as \rf{state-closed-loop-SO},
then $\bar\cX=\h X$, which implies that $\bar\cX$ satisfies \rf{state-u2-star11}.
By the uniqueness of the solution to \rf{state-u2-star11}, we get $\bar X=\bar\cX$ immediately.

\ms
Now let us show that \rf{LLQ-fSDE1} can be really rewritten as \rf{state-closed-loop-SO}.
Indeed, by the definitions \rf{def-barA}--\rf{def-barR} of $\cA$, $\cQ$, $\cS_1$ and $\cS_3$, we get
\begin{align}
&\cS_1^\top\cR^{-1}\cB^\top+\cS_1^\top\cR^{-1}\cS_1\Si
-\cA^\top-\cQ\Si+(\cS_3^\top-\cS_1^\top\cR^{-1}\cS_2)\nn\\
&\qq\times\h\Si^{-1}(\Si\cC^\top-\Si\cS_2^\top\cR^{-1}\cB^\top-\Si\cS_2^\top\cR^{-1}\cS_1\Si+\Si\cS_3\Si)\nn\\
&\q=A-B_1\h R_1^{-1} (B^\top_1P_1+D_1^\top P_1C)+B_1\h R_1^{-1}B_1^\top\Si+[ B_2-B_1\h R_1^{-1}D_1^\top P_1 D_2]\nn\\
&\qq\times[\cR^{-1}\cB^\top+\cR^{-1}\cS_1\Si]-[B_1\h R_1^{-1}D_1^\top+ (B_2-B_1\h R_1^{-1}D_1^\top P_1 D_2)\cR^{-1}\cS_2]\nn\\
&\qq\times\h\Si^{-1}(\Si\cC^\top-\Si\cS_2^\top\cR^{-1}\cB^\top-\Si\cS_2^\top\cR^{-1}\cS_1\Si+\Si\cS_3\Si)\nn\\
&\q=A +B_1\h\Th_1  +B_2\h\Th_2,\label{thm-close-system-proof1}
\end{align}
where $\h\Th_1$ and $\h\Th_2$  are defined by \rf{thm:closed-loop-system-SO1}--\rf{thm:closed-loop-system-SO2}.
In a similar way, we  also have
\begin{align}
&\cS_2^\top\cR^{-1}\cB^\top-\cC^\top
-(\cS_3-\cS_2^\top\cR^{-1}\cS_1)\Si+(\cN-\cS_2^\top\cR^{-1}\cS_2)\nn\\
&\qq~\times\h\Si^{-1}(\Si\cC^\top-\Si\cS_2^\top\cR^{-1}\cB^\top-\Si\cS_2^\top\cR^{-1}\cS_1\Si+\Si\cS_3\Si)\nn\\
&\q=C +D_1\h\Th_1  +D_2\h\Th_2. \label{thm-close-system-proof2}
\end{align}
By \rf{thm-close-system-proof1} and \rf{thm-close-system-proof2},
we see that equation \rf{LLQ-fSDE1} can be rewritten as \rf{state-closed-loop-SO}.

\ms
When $\xi=0$ and $\si=0$, the unique solution of BSDE \rf{BLQ-BSDE1} is given by $(\f,\b)\equiv(0,0)$.
From the representation \rf{U-value1} of the value function $U$, we get
$$
U(0,0;\bar u_2)=-\lan\Si(0)x,x\ran.
$$
Substituting the above into \rf{cost-u2-U} yields (noting $U(0;u_2)=U(0,0;u_2)$)
$$
J(x;\h u_1,\h u_2)=J(x;\bar\a_1(\bar u_2,x),\bar u_2)=\lan(P_1(0)-\Si(0))x,x\ran.
$$
This completes the proof.
\end{proof}

\begin{remark}
It is noteworthy that the results obtained in Subsections \ref{sub:BLQ-t} and \ref{subsec:CR}
still hold true if the assumption \ref{ass:H5} is replaced by  \rf{uniform-concave-U},
because \rf{uniform-concave-U} is  sufficient for the well-posedness of Riccati equation \rf{BLQ-Riccati-Equation1}.
\end{remark}

\section{Connections between Problems (SG) and (NG)}\label{sec:Connection}
Recall from \cite{Sun2020} that the  (UCC) condition \ref{ass:H3} and \ref{ass:H5}
is  sufficient and almostly necessary for the solvability of Problem (NG).
In this section, under \ref{ass:H3} and \ref{ass:H5},
we shall establish some interesting connections between Problem (SG) and Problem (NG).

\subsection{Relatioship between the Riccati equations}\label{Subsec:Ri-E}
The Riccati equation associated with Problem (NG)  reads
\bel{Ri-E}\left\{\begin{aligned}
  & \dot P+PA+A^\top P+C^\top PC+Q \\
  & \hp{\dot P} -(PB+C^\top PD)(R+D^\top PD)^{-1}(B^\top P+D^\top PC)=0, \\
  & P(T)=G,
\end{aligned}\right.\ee
where
\begin{align}
&B=(B_1,B_2),\q  D=(D_1,D_2),
\q R= \begin{pmatrix} R_{1} & 0\\ 0 & R_{2}\end{pmatrix}.
%
%
\end{align}

\begin{definition}\label{def-solution-RiE}
An absolutely continuous function $P:[0,T]\to \dbS^n$ is called a  {\it  solution} of Riccati equation \rf{Ri-E} if
\begin{enumerate}[(i)]
\item  $P$ satisfies \rf{Ri-E}  almost everywhere on $[0,T]$, and

\item  $R+D^\top PD$ is invertible with $(R+D^\top PD)^{-1}\in L^\i(0,T;\dbS^n)$.
\end{enumerate}
\end{definition}

In \cite[Definition 4.2]{Sun2020}, the solution $P$ of \rf{Ri-E}  is called a {\it strongly regular solution}
if it also satisfies:
\bel{SRS}
(-1)^{i+1}(R_i+D_i^\top PD_i)\gg 0,\q i=1,2.
\ee
However, the uniformly positive definiteness \rf{SRS} does not imply the open-loop solvability of Problem (NG),
which is different from the situation in control problems (see \cite[Example 4.5]{Sun2020}).
Thus, the  condition \rf{SRS} is only used to ensure the invertibility of the singular term $R+D^\top PD$
(i.e., the property (ii) in \autoref{def-solution-RiE}).

\ms
The following result establishes a connection between Riccati equations \rf{Ri-E},
\rf{Ric-1} and \rf{BLQ-Riccati-Equation1}, which are introduced for solving Problems (NG),
(FLQ) and (LLQ), respectively.

\begin{theorem}\label{Thm:Riccati}
Let {\rm\ref{ass:H1}--\ref{ass:H3}} and {\rm\ref{ass:H5}}  hold. Then Riccati equation \rf{Ri-E} admits a unique  solution
\bel{Thm:Riccati-main}
P=P_1-\Si,
\ee
where $P_1\in C([0,T];\dbS^n)$ and $\Si\in C([0,T];\dbS^n_-)$ are the unique solutions to Riccati equations
\rf{Ric-1} and \rf{BLQ-Riccati-Equation1}, respectively.
\end{theorem}

\begin{remark}\label{remark-Riccati}
We emphasize that \autoref{Thm:Riccati} still holds true if
\ref{ass:H5} is replaced by \rf{uniform-concave-U},
because  Riccati equation \rf{BLQ-Riccati-Equation1} is still solvable under  \rf{uniform-concave-U}.
From \autoref{Prop:saddle-point-control}, we see that the conditions \ref{ass:H3} and \rf{uniform-concave-U}
are strictly weaker than the assumptions \ref{ass:H3} and \ref{ass:H5}, which were imposed in \cite[Theorem 4.3]{Sun2020}.
Thus, by \autoref{Thm:Riccati}, first, we establish a connection between the Riccati equations \rf{Ri-E},
\rf{Ric-1} and \rf{BLQ-Riccati-Equation1};
second, we prove the well-posedness of Riccati equation  \rf{Ri-E} with a new constructive method;
third,  the assumptions imposed in \cite[Theorem 4.3]{Sun2020} are relaxed.
\end{remark}

To prove \autoref{Thm:Riccati}, we  need to make some preparations.
The difficulty mainly comes from the singularity of Riccati equations  \rf{Ri-E},
 \rf{Ric-1} and \rf{BLQ-Riccati-Equation1}.
Recall from \autoref{prop:follower} that under \ref{ass:H3},
Riccati equation \rf{Ric-1} admits a unique solution $P_1\in C([0,T];\dbS^n)$ satisfying
\bel{Ri-up}
R_1+D_1^\top P_1D_1\gg 0.
\ee
Recall from \autoref{pro:Riccati-BLQ} that under \ref{ass:H5},
Riccati equation \rf{BLQ-Riccati-Equation1} admits a unique solution  $\Si\in C([0,T];\dbS_-^n)$
such that $\h\Si= I+\Si\cN-\Si\cS_2^\top\cR^{-1}\cS_2$ is invertible with
\bel{h-Si-inverse}
\h\Si^{-1}=(I+\Si\cN-\Si\cS_2^\top\cR^{-1}\cS_2)^{-1}\in L^\i(0,T;\dbR^n)
\q\hbox{and}\q \h\Si^{-1}\Si\in L^\i(0,T;\dbS^n).
\ee
We also recall the definitions \rf{def-barA}--\rf{def-barR} of $\cA,\,\cB,\,\cC$ and $\cQ,\,\cN,\,\cR,\,\cS_i\,(i=1,2,3)$.
Combining \rf{Ri-up} with the fact $\Si\les0$, we have
\bel{R1-Si-inverse}
R_{1}+D_1^\top [P_1-\Si]D_1\gg0.
\ee
Denote
\begin{align}
\Phi&= R_2+D_2^\top (P_1-\Si)D_2
-D_2^\top(P_1-\Si)D_1[R_1+D_1^\top(P_1-\Si)D_1]^{-1}D_1^\top (P_1-\Si)D_2,\label{def-Phi}\\
%
%
%
%
\h\Phi&=\cR^{-1}+\cR^{-1}\cS_2\h\Si^{-1}\Si\cS_2^\top\cR^{-1}.\label{def-hPhi}
%
%
\end{align}

\begin{lemma}\label{lem:inverse}
The matrix-valued function $\Phi$ is invertible with its inverse given by $\Phi^{-1}=\h\Phi$.
\end{lemma}

By \autoref{lem:inverse}, it is straightforward to see that the matrix
$$
R+D^\top (P_1-\Si)D= \begin{pmatrix} R_{1}+D_1^\top (P_1-\Si)D_1 & D_1^\top (P_1-\Si)D_2 \\
D_2^\top (P_1-\Si)D_1 & R_{2}+D_2^\top (P_1-\Si)D_2\end{pmatrix},
$$
is invertible with its inverse given by
\bel{RDPD-inverse}
[R+D^\top (P_1-\Si)D]^{-1}= \begin{pmatrix}\cM_{11} & \cM_{12} \\
\cM_{21} & \cM_{22}\end{pmatrix},
\ee
where
\begin{align}
\cM_{11}&= [R_1+D_1^\top (P_1-\Si)D_1]^{-1}+[R_1+D_1^\top (P_1-\Si)D_1]^{-1}D_1^\top(P_1-\Si)D_2\nn\\
&\q\times \h\Phi D_2^\top(P_1-\Si)D_1[R_1+D_1^\top (P_1-\Si)D_1]^{-1},\nn\\
\cM_{12}&=\cM^\top_{21}=-[R_1+D_1^\top (P_1-\Si)D_1]^{-1}D_1^\top(P_1-\Si)D_2\h\Phi,\q \cM_{22}=\h\Phi.
\label{RDPD-inverse1}
\end{align}

\begin{remark}
\autoref{lem:inverse} serves as a crucial bridge between the singular terms of
Riccati equations \rf{Ri-E}, \rf{Ric-1} and \rf{BLQ-Riccati-Equation1}.
The construction of this bridge is technical.
Once the explicit form of $\Phi^{-1}$ is derived,
the result can be proved by a lengthy verification.
We sketch the proof in Appendix.
\end{remark}

\subsection{Proof of \autoref{Thm:Riccati}}
{\it\bf\large Uniqueness:}
Suppose that $ P, \bar P\in C([0,T];\dbS^n)$ are two solutions of \rf{Ri-E}.
Then by \autoref{def-solution-RiE}, both $R+D^\top PD$ and $R+D^\top\bar PD$ are invertible
with their inverses belonging to $L^\i(0,T;\dbS^n)$. Denote $\D P= P-\bar P$.
Then $\D P$ satisfies the following linear ordinary differential equation:
\bel{Delta-Ric1}\left\{\begin{aligned}
  & \D\dot P+\D PA+A^\top\D P+C^\top \D PC -(\D PB+C^\top\D PD)(R+D^\top PD)^{-1} \\
  & \hp{\dot P}\times(B^\top  P+D^\top PC)+(\bar PB^\top +C^\top\bar PD)(R+D^\top PD)^{-1} D^\top\D PD\\
  & \hp{\dot P}\times(R+D^\top\bar PD)^{-1}(B^\top  P+D^\top PC)-(\bar PB^\top +C^\top\bar PD)\\
  &\hp{\dot P} \times(R+D^\top\bar PD)^{-1}(B^\top \D P+D^\top\D PC)=0,\\
  &\D P(T)=0.
\end{aligned}\right.\ee
Note that $ P$, $\bar P$, $(R+D^\top PD)^{-1}$ and $(R+D^\top\bar PD)^{-1}$ are all bounded.
Then by a standard argument using the Gr\"{o}nwall's inequality, we get $\D P\equiv 0$.
It follows that Riccati equation \rf{Ri-E} admits at most one solution.

\bs\no
{\it\bf\large Existence:}
Note that $P_1(T)-\Si(T)=G$ and
\begin{align}
\dot P_1-\dot\Si
 %
%
%
%
%
%
%
%
%
%
&= -(P_1-\Si)A-A^\top (P_1-\Si)-C^\top (P_1-\Si)C-Q \nn\\
&\q+(P_1B_1+C^\top P_1D_1)\h R_1^{-1}(B_1^\top P_1+D_1^\top P_{1}C)-\Si B_1\h R_1^{-1}(D_1^\top P_1C+B_1^\top P_1)\nn\\
&\q-(P_1B_1+C^\top P_1D_1)\h R_1^{-1}B_1^\top\Si-C^\top\Si C-\Si \cQ\Si\nn\\
&\q+[\cC +\Si\cS_3^\top]\h\Si^{-1}\Si[\cC^\top-\cS_2^\top\cR^{-1}\cB^\top-\cS_2^\top\cR^{-1}\cS_1\Si +\cS_3\Si]\nn\\
&\q -[\cB\cR^{-1}\cS_2+\Si\cS^\top_1\cR^{-1}\cS_2]\h\Si^{-1}\Si[\cC^\top +\cS_3\Si]+\cB\h\Phi\cB^\top
+\Si\cS_1^\top\h\Phi\cB^\top\nn\\
&\q +\cB\h\Phi\cS_1\Si+\Si\cS_1^\top\h\Phi\cS_1\Si\nn\\
&\deq -(P_1-\Si)A-A^\top (P_1-\Si)-C^\top (P_1-\Si)C-Q +(I).\label{P1-Si}
\end{align}
Comparing the above with \rf{Ri-E}, to prove that $P_1-\Si$ satisfies Riccati equation \rf{Ri-E},
it suffices to show
\begin{align}
F&\deq (I)-[(P_1-\Si)B_1+C^\top(P_1-\Si)D_1]\cM_{11} [B_1^\top(P_1-\Si)+D_1^\top(P_1-\Si)C]\nn\\
&\q-[(P_1-\Si)B_1+C^\top(P_1-\Si)D_1]\cM_{12} [B_2^\top(P_1-\Si)+D_2^\top(P_1-\Si)C]\nn\\
&\q-[(P_1-\Si)B_2+C^\top(P_1-\Si)D_2]\cM_{21} [B_1^\top(P_1-\Si)+D_1^\top(P_1-\Si)C]\nn\\
&\q-[(P_1-\Si)B_2+C^\top(P_1-\Si)D_2]\cM_{22} [B_2^\top(P_1-\Si)+D_2^\top(P_1-\Si)C]\nn\\
&=0,\label{def-F}
\end{align}
where $\cM_{i,j}$ $(i,j=1,2)$ is defined by \rf{RDPD-inverse1}.
By the definitions \rf{def-barA} and \rf{def-barR}, the function $F$ can be rewritten as
\begin{align}
F&=  P_1 B_1(f_1)+\Si B_1(f_2)+ P_1 B_2(f_3)+ \Si B_2(f_4)+ C^\top P_1D_1(f_5)+ C^\top P_1D_2(f_6)+(f_7),
\label{def-Fi}
\end{align}
where
\begin{align}
(f_2)&= -(f_1),\q (f_4)= -(f_3),\q (f_5)= (f_1),\q (f_6)= (f_3),\label{def-f-2456}
\end{align}
and
\begin{align}
(f_1)&= \h R_1^{-1}(D_1P_1C+B_1^\top P_1)-\h R_1^{-1}B_1^\top\Si+\h R_1^{-1}D_1^\top\h\Si^{-1}
\Si[\cC^\top-\cS_2^\top\cR^{-1}\cB^\top-\cS_2^\top\cR^{-1}\cS_1\Si\nn\\
&\q +\cS_3\Si]
-\h R_1^{-1}D_1^\top P_1D_2\cR^{-1}\cS_2\h\Si^{-1}\Si[\cC^\top+\cS_3\Si]+\h R_1^{-1}D_1^\top P_1D_2\h\Phi
[\cB^\top+\cS_1\Si]\nn\\
&\q-\cM_{11}[B_1^\top(P_1-\Si)+D_1^\top(P_1-\Si)C]-\cM_{12} [B_2^\top(P_1-\Si)+D_2^\top(P_1-\Si)C],\label{def-f1}\\
(f_3)&= \cR^{-1}\cS_2\h\Si^{-1}\Si[\cC^\top+\cS_3\Si]-\h\Phi\cB^\top -\h\Phi\cS_1\Si
-\cM_{21} [B_1^\top(P_1-\Si)+D_1^\top(P_1-\Si)C]\nn\\
&\q-\cM_{22} [B_2^\top(P_1-\Si)+D_2^\top(P_1-\Si)C],\label{def-f3}\\
(f_7)&= C^\top \Si D_1\cM_{11} [B_1^\top(P_1-\Si)+D_1^\top(P_1-\Si)C]+C^\top \Si D_1\cM_{12} [B_2^\top(P_1-\Si)\nn\\
&\q+D_2^\top(P_1-\Si)C]+C^\top \Si D_2\cM_{21} [B_1^\top(P_1-\Si)+D_1^\top(P_1-\Si)C]\nn\\
&\q +C^\top \Si D_2\cM_{22} [B_2^\top(P_1-\Si)+D_2^\top(P_1-\Si)C]-C^\top\Si C-C^\top\h\Si^{-1}\Si\cC^\top\nn\\
&\q+C^\top\h\Si^{-1}\Si\cS_2\cR^{-1}\cB^\top+C^\top\h\Si^{-1}\Si\cS_2^\top\cR^{-1}\cS_1\Si-C^\top\h\Si^{-1}\Si\cS_3\Si.
\label{def-f7}
\end{align}
Thus to prove $F=0$, we only need to show $(f_i)=0;i=1,3,7$. In the following, we shall prove them separately.

\ms
(1)\textbf{ Proof of $(f_3)=0$.} By the definitions of $\h\Phi$, $\cM_{21}$ and $\cM_{22}$, $(f_3)$ can be rewritten as
\begin{align}
-(f_3)&=\cR^{-1}\cS_2\h\Si^{-1}\Si D_1(R_1+D_1^\top P_1 D_1)^{-1}B^\top_1\Si-\h\Phi D_2^\top P_1D_1(R_1+D_1^\top P_1 D_1)^{-1}B_1^\top\Si\nn\\
&\q-\h\Phi D_2^\top(P_1-\Si)D_1[R_1+D_1^\top(P_1-\Si)D_1]^{-1}[B_1^\top(P_1-\Si)+D_1^\top(P_1-\Si)C]\nn\\
&\q+\h\Phi D_2^\top P_1D_1(R_1+D_1^\top P_1 D_1)^{-1}(D_1^\top P_1C+B_1^\top P_1)-\h\Phi D_2^\top\Si C\nn\\
&\q-\cR^{-1}\cS_2\h\Si^{-1}\Si[D_1(R_1+D_1^\top P_1 D_1)^{-1}(D_1^\top P_1C+B_1^\top P_1)-C]\nn\\
&=-(f_{3})_1 B_1^\top(P_1-\Si)-(f_3)_2C,\label{f3}
\end{align}
where
\begin{align}
(f_3)_1&= \h\Phi D_2^\top(P_1-\Si)D_1[R_1+D_1^\top(P_1-\Si)D_1]^{-1}-\h\Phi D_2^\top P_1D_1(R_1+D_1^\top P_1 D_1)^{-1}\nn\\
&\q+\cR^{-1}\cS_2\h\Si^{-1}\Si D_1(R_1+D_1^\top P_1 D_1)^{-1},\label{f3-1}\\
(f_3)_2&= \h\Phi D_2^\top(P_1-\Si)D_1[R_1+D_1^\top(P_1-\Si)D_1]^{-1}D_1^\top(P_1-\Si)+\h\Phi D_2^\top\Si\nn\\
&\q-\cR^{-1}\cS_2\h\Si^{-1}\Si-\h\Phi D_2^\top P_1D_1(R_1+D_1^\top P_1 D_1)^{-1}D_1^\top P_1\nn\\
&\q+\cR^{-1}\cS_2\h\Si^{-1}\Si D_1(R_1+D_1^\top P_1 D_1)^{-1}D_1^\top P_1.\label{f3-2}
\end{align}
Then from $(f_3)_1=0$ and $(f_3)_2=0$ (see Appendix for the proof), we get $(f_3)=0$.

\ms
(2) \textbf{ Proof of $(f_1)=0$.} We can rewrite $(f_1)$ as
\begin{align}
(f_1)&= -(f_1)_1 [B_2^\top(P_1-\Si)+D_2^\top P_1C]-(f_1)_2[ B_1^\top(P_1-\Si)+D_1^\top P_1 C]-(f_1)_3,
\label{f1}
\end{align}
where $(f_1)_1= -(f_{3})_1^\top=0$ and
\begin{align}
(f_1)_2&= [R_1+D_1^\top(P_1-\Si)D_1]^{-1}+[R_1+D_1^\top(P_1-\Si)D_1]^{-1} D_1^\top(P_1-\Si)D_2^\top\h\Phi D_2^\top\nn\\
&\q\times(P_1-\Si)D_1[R_1+D_1^\top(P_1-\Si)D_1]^{-1}-\h R_1^{-1}+\h R_1^{-1}D_1^\top P_1 D_2\cR^{-1}\cS_2\h\Si^{-1}\Si D_1\h R_1^{-1}\nn\\
&\q-\h R_1^{-1}D_1^\top\h\Si^{-1}\Si D_1\h R_1^{-1}+\h R_1^{-1}D_1^\top\h\Si^{-1}\Si\cS_2^\top\cR^{-1} D_2^\top P_1D_1\h R_1^{-1}\nn\\
&\q-\h R_1^{-1}D_1^\top R_1D_2\h\Phi D_2^\top P_1D_1\h R_1^{-1},\label{f12}\\
(f_1)_3&= -\cM_{11}D_1^\top\Si C-\cM_{12}D_2^\top\Si C-\h R_1^{-1} D_1^\top P_1D_2\cR^{-1}\cS_2\h\Si^{-1}\Si C+\h R_1^{-1}D_1^\top\h\Si^{-1}\Si C.
\label{f13}
\end{align}
Then from $(f_1)_2=0$ and $(f_1)_3=0$ (see Appendix for the proof), we get $(f_1)=0$.

\ms
(3) \textbf{ Proof of $(f_7)=0$.}  We can rewrite $(f_7)$ as
\begin{align}
(f_7)=(f_7)_1 B_1 P_1+(f_7)_2 B_1 \Si+(f_7)_3B_2 P_1+(f_7)_4B_2\Si+(f_7)_5,
\end{align}
where
\begin{align}
(f_7)_1&=-(f_7)_2=-(f_1)_3^\top=0,\nn\\
(f_7)_3&=-(f_7)_4=C^\top (f_3)_2^\top=0,\nn\\
(f_7)_5&= C^\top \Si D_1\cM_{11} D_1^\top(P_1-\Si)C+C^\top \Si D_1\cM_{12} D_2^\top(P_1-\Si)C\nn\\
%
&\q+C^\top \Si D_2\cM_{21}D_1^\top(P_1-\Si)C+C^\top \Si D_2\cM_{22} D_2^\top(P_1-\Si)C\nn\\
&\q -C^\top\Si C-C^\top\h\Si^{-1}\Si[D_1\h R_1^{-1}D_1^\top P_1C-C]\nn\\
&\q+C^\top\h\Si^{-1}\Si\cS_2\cR^{-1}[D_2^\top P_1D_1\h R_1^{-1}D_1^\top P_1C- D_2^\top P_1C].\label{f7-5}
\end{align}
Then from $(f_7)_5=0$ (see Appendix for the proof), we get $(f_7)=0$.
$\hfill\qed$

\begin{remark}
From the above proof, we see that \autoref{Thm:Riccati} can be proved by comparing \rf{P1-Si} with \rf{Ri-E}.
Although the bridge between the singular terms of Riccati equations \rf{Ri-E}, \rf{Ric-1} and \rf{BLQ-Riccati-Equation1}
has been established by \autoref{lem:inverse},
the verification  is still technical and lengthy.
For more details of the proof, please see Appendix.
\end{remark}

\subsection{Equivalence between  Stackelberg equilibria and open-loop saddle points}\label{Sec:SPG}
In \autoref{Thm:Riccati}, a connection between the Riccati equations associated with Problems (SG) and (NG)
has been established. In this subsection, we shall show that the Stackelberg equilibrium,
obtained in \autoref{Thm:SOP-control1}, exactly is the unique open-loop saddle point of Problem (NG).

\begin{theorem}\label{Thm:saddle-point-control}
Suppose that {\rm\ref{ass:H1}--\ref{ass:H3}} and {\rm\ref{ass:H5}} hold.
Then the following results hold.
\begin{enumerate}[(i)]
\item The Stackelberg equilibrium   $(\h u_1,\h u_2)\in\cU_1\times\cU_2$ of {\rm Problem (SG)},
obtained in  {\rm\autoref{Thm:SOP-control1}},
is the unique open-loop saddle point of {\rm Problem (NG)}.

\item The value function of {\rm Problem (NG)} is given by
\bel{V-SPG}
V(x)=\lan(P_1(0)-\Si(0))x,x\ran=\lan P(0)x,x\ran, \q \forall x\in\dbR^n.
\ee
\end{enumerate}
\end{theorem}

\begin{proof}
(i) Under \ref{ass:H5}, by \cite[Theorem 4.4]{Sun2020} we get that Problem (NG)
admits a unique open-loop saddle point $(u_1^*,u_2^*)\in\cU_1\times\cU_2$.
By the definition of open-loop saddle points, we have
\bel{cJ-line}
\inf_{u_1\in\cU_1}\cJ_{u^*_2}(x;u_1)=\cJ_{u^*_2}(x;u^*_1) \q\hbox{and}\q J(x;u^*_1,u_2^*)=\sup_{u_2\in\cU_2}J(x;u^*_1,u_2),
\ee
where $\cJ_{u_2^*}$ is defined by \rf{cost-follower}. Then by \autoref{prop:follower}, we get
\bel{}
 u^*_1(s)=\bar\a_1(s;u^*_2,x),\q s\in[0,T],
\ee
where $\bar\a_1$ is defined by \rf{closed-loop-repre-X}.
Thus
\bel{closed1}
J(x;u^*_1,u_2^*)=J(x;\bar\a_1(u^*_2,x),u^*_2).
\ee
Moreover, recall from \autoref{prop:follower} that
\bel{closed111}
J(x;\bar\a_1(u_2,x),u_2)=\inf_{u_1\in\cU_1} J(x;u_1,u_2),\q\forall u_2\in\cU_2.
\ee
Then by the second equality in \rf{cJ-line}, we get
$$
J(x;u^*_1,u_2^*)\ges J(x;u_1^*,u_2)\ges J(x;\bar\a_1(u_2,x),u_2),\q\forall u_2\in\cU_2.
$$
Combining the above with \rf{closed1} yields that
\bel{u2-star1}
J(x;\bar\a_1(u^*_2,x),u^*_2)\ges J(x;\bar\a_1(u_2,x),u_2),\q\forall u_2\in\cU_2.
\ee
In other words, $u_2^*$ is an optimal control of  Problem (LLQ),
which is the leader's  problem.

\ms
On the other hand, by \autoref{prop:leader1},
Problem (LLQ) admits a unique optimal control $\bar u_2=\h u_2$ under \ref{ass:H1}--\ref{ass:H3} and \ref{ass:H5}.
Thus, we must have $u_2^*=\h u_2$.
Combining this with the facts $u^*_1=\bar\a_1(u^*_2,x)$ and $\h u_1=\bar\a_1(\h u_2,x)$,
we get $u_1^*=\h u_1$. It follows that  $(\h u_1,\h u_2)\in\cU_1\times\cU_2$
is the unique open-loop saddle of Problem (NG).

\ms
(ii) By the definition of the value function of Problem (NG) and the fact $(u_1^*,u_2^*)=(\h u_1,\h u_2)$, we get
$$
V(x)=J(x;u_1^*,u_2^*)=J(x;\h u_1,\h u_2).
$$
Then from \rf{SOP-V} and \rf{Thm:Riccati-main}, we obtain \rf{V-SPG}.
\end{proof}

\begin{remark}
We emphasize again that the weak (UCC) condition and the (UCC) condition are almost necessary
for the existence of a Stackelberg equilibrium and the existence of an open-loop saddle point, respectively.
Then from \autoref{Thm:SOP-control}, \autoref{example2} and \autoref{Thm:saddle-point-control},
we conclude that the gap between the weak (UCC) condition (i.e., \ref{ass:H3}--\ref{ass:H4})
and the (UCC) condition (i.e., \ref{ass:H3} and \ref{ass:H5}) is the main reason causing the different performances
between Problems (SG) and (NG).
\end{remark}

\ms

Denote
\bel{Th-star}
(\Th_1^{*\top},\Th_2^{*\top})^\top=-\big(R+D^\top PD\big)^{-1}\big(B^\top P+D^\top PC\big).
\ee

\ms

\begin{theorem}\label{closed-system-SPG}
Let {\rm\ref{ass:H1}}--{\rm\ref{ass:H3}} and \rf{uniform-concave-U} hold.
Then the Stackelberg equilibrium $(\h u_1,\h u_2)$ of {\rm Problem (SG)}
can be represented as:
\begin{align}\label{star-u}
\h u_1=u_1^*\equiv \Th^*_1 X^* \q\hbox{and }\q \h u_2=u_2^*\equiv \Th^*_2 X^*,
\end{align}
where $ X^*$ is the unique solution of the closed-loop system:
\bel{state-closed-loop-SPG}\left\{\begin{aligned}
   d X^*(s) &=\big\{A(s) X^*(s) +B_1(s)\Th^*_1 (s)X^*(s) +B_2(s)\Th^*_2(s) X^*(s) \big\}ds\\
         &~\hp{=}+\big\{C(s) X^*(s) +D_1(s)\Th^*_1(s)X^*(s)+ D_2(s)\Th^*_2(s) X^*(s)\big\}dW(s), \\
     X^*(0) &= x.
\end{aligned}\right.\ee
If {\rm\ref{ass:H5}} also holds, then $(u_1^*,u_2^*)$ is the unique open-loop saddle point of {\rm Problem (NG)}.
\end{theorem}

\begin{proof}
Using the similar argument to that employed in \autoref{Thm:Riccati},  
we can get
\begin{align}
\h\Th_1&=-\cM_{11}[B_1^\top (P_1-\Si)+D_1^\top(P_1-\Si)C]-\cM_{12}[B_2^\top (P_1-\Si)+D_2^\top(P_1-\Si)C],\nn\\
\h\Th_2&= -\cM_{21}[B_1^\top (P_1-\Si)+D_1^\top(P_1-\Si)C]-\cM_{22}[B_2^\top (P_1-\Si)+D_2^\top(P_1-\Si)C],\label{hat-Th}
\end{align}
where $\h\Th_i$ and $\cM_{ij}$  $(i,j=1,2)$  are defined by
\rf{thm:closed-loop-system-SO1}--\rf{thm:closed-loop-system-SO2} and \rf{RDPD-inverse1}, respectively.
Then by the fact $P=P_1-\Si$ obtained in \autoref{Thm:Riccati}, we can rewrite \rf{hat-Th} as
\bel{}
(\h \Th_1^{\top},\h \Th_2^{\top})^\top=(\Th_1^{*\top},\Th_2^{*\top})^\top=-\big(R+D^\top PD\big)^{-1}\big(B^\top P+D^\top PC\big).
\ee
Thus,  the Stackelberg equilibrium $(\h u_1,\h u_2)$ obtained in \autoref{Thm:SOP-control1} can be rewritten as \rf{star-u}.
If the additional assumption \ref{ass:H5} holds, by \autoref{Thm:saddle-point-control}, the control pair $(u_1^*,u_2^*)$ is the unique open-loop
saddle point of Problem (NG).
\end{proof}

\begin{remark}
By \autoref{closed-system-SPG}, we show that under {\rm\ref{ass:H1}}--{\rm\ref{ass:H3} and \rf{uniform-concave-U}},
the Stackelberg equilibrium of Problem (SG) admits another closed-loop representation \rf{star-u},
in terms of the solution to Riccati equation  \rf{Ri-E}.
When \ref{ass:H5} also holds,  \rf{state-closed-loop-SPG} coincides with the closed-loop system of Problem (NG),
which was given  in \cite[Theorem 4.4]{Sun2020}.
\end{remark}


\section{Conclusion}\label{Sec:Conclusion}
In conclusion, we show that under the weak (UCC) condition (i.e, \ref{ass:H3}--\ref{ass:H4}),
a Stackelberg equilibrium of  Problem (SG)  can be explicitly obtained
by solving a forward-backward  stochastic LQ optimal control problem (see \autoref{Thm:SOP-control}).
Interestingly, under the stronger (UCC) condition (i.e, \ref{ass:H3} and \ref{ass:H5}),
the Stackelberg equilibrium of  Problem (SG) exactly is the unique open-loop saddle point of Problem (NG)
(see \autoref{Thm:saddle-point-control} and \autoref{closed-system-SPG}).
It follows that the open-loop saddle point of Problem (NG) can be obtained by considering the game
in a leader-follower manner, which is a little surprising.
These results are achieved by a careful investigation of backward stochastic LQ optimal control problems
(see \autoref{Prop:saddle-point-control} and \autoref{Prop:backward1}).
Moreover, an explicit relationship between the Riccati equations associated with Problem (NG)
(i.e., \rf{Ri-E}) and Problem (SG) (i.e., \rf{Ric-1} and \rf{BLQ-Riccati-Equation1}) is established
(see \autoref{Thm:Riccati}). Indeed, we show that \rf{BLQ-Riccati-Equation1} serves as
a bridge between the Riccati equations associated with stochastic LQ optimal controls
and two-person zero-sum stochastic LQ Nash games (i.e., Problems (FLQ) and (NG)).
As a byproduct, the well-posedness of Riccati equation \rf{Ri-E} is reestablished by a completely new method,
which can help to relax the assumptions imposed by Sun \cite{Sun2020}.

\section{Appendix}\label{Sec:Appendix}
\subsection{Proof of \autoref{lem:inverse}}\label{proof-lemma}
By the definition of $\cR$, it is straightforward to see that
\begin{align}
\nn\Phi
&=\cR-D_2^\top \Si D_2+D_2^\top \Si D_1[R_1+D_1^\top (P_1-\Si)D_1]^{-1}D_1^\top P_1D_2\\
&\q -D_2^\top P_1D_1[R_1+D_1^\top (P_1-\Si)D_1]^{-1}D_1^\top\Si D_1[R_1+D_1^\top P_1D_1]^{-1}D_1^\top P_1D_2\nn\\
&\q+D_2^\top P_1 D_1[R_1+D_1^\top (P_1-\Si)D_1]^{-1}D_1^\top\Si D_2\nn\\
\label{Phi-1}
&\q-D_2^\top \Si D_1[R_1+D_1^\top (P_1-\Si)D_1]^{-1}D_1^\top\Si D_2.
\end{align}
Then by the definition of $\cS_2$, we get
\begin{align}
\Phi\h\Phi&=I+[D_2^\top-D_2^\top P_1D_1\h R_1^{-1}D_1^\top]\h\Si^{-1} \Si[D_2-D_1\h R_1^{-1}D_1^\top P_1D_2]\cR^{-1}\nn\\
\nn&\q+\big\{D_2^\top \Si D_1[R_1+D_1^\top (P_1-\Si)D_1]^{-1}D_1^\top P_1D_2-D_2^\top \Si D_2\\
&\qq+D_2^\top P_1 D_1[R_1+D_1^\top (P_1-\Si)D_1]^{-1}D_1^\top\Si [D_2-D_1\h R_1^{-1}D_1^\top P_1D_2]\nn\\
\nn
&\qq-D_2^\top \Si D_1[R_1+D_1^\top (P_1-\Si)D_1]^{-1}D_1^\top\Si D_2\big\}\\
&\q\times\big\{\cR^{-1}+\cR^{-1}[D_2^\top-D_2^\top P_1D_1\h R_1^{-1}D_1^\top]
\h\Si^{-1} \Si[D_2-D_1\h R_1^{-1}D_1^\top P_1D_2]\cR^{-1}\big\}\nn\\
&\deq I+(a)\cR^{-1}.\label{a}
\end{align}
%
%
%
%
%
%
%
%
%
%
%
%
%
Thus to prove \autoref{lem:inverse}, noting that $\Phi$ and $\h\Phi$ are symmetric,
it suffices to show that $(a)=0$. By the definition \rf{hat-Si1} of $\h\Si$, the function $(a)$ can be simplified as follows:
\begin{align}
(a)
%
%
%
%
%
%
%
%
%
%
%
%
%
&=\big\{D_2^\top-D_2^\top P_1D_1\h R_1^{-1}D_1^\top+D_2^\top P_1D_1[R_1+D_1^\top (P_1-\Si)D_1]^{-1}D_1^\top -D_2^\top P_1D_1\nn\\
\nn&\qq\times[R_1+D_1^\top (P_1-\Si)D_1]^{-1}D_1^\top\Si D_1\h R_1^{-1}D_1^\top\big\}\h\Si^{-1} \Si[D_2-D_1\h R_1^{-1}D_1^\top P_1D_2]\\
\nn&\q+\big\{D_2^\top \Si D_1[R_1+D_1^\top (P_1-\Si)D_1]^{-1}D_1^\top P_1D_2-D_2^\top \Si D_2\\
\nn
&\qq-D_2^\top \Si D_1[R_1+D_1^\top (P_1-\Si)D_1]^{-1}D_1^\top\Si D_2\big\}\\
\nn&\q+\big\{D_2^\top \Si D_1[R_1+D_1^\top (P_1-\Si)D_1]^{-1}D_1^\top P_1D_2-D_2^\top \Si D_1[R_1+D_1^\top (P_1-\Si)D_1]^{-1}\\
%
%
%
&\qq\times D_1^\top\Si D_2-D_2^\top \Si D_2\big\}\cR^{-1}[D_2^\top-D_2^\top P_1D_1\h R_1^{-1}D_1^\top]\h\Si^{-1}
\Si[D_2-D_1\h R_1^{-1}D_1^\top P_1D_2].\label{(a)-1}
\end{align}
Further,  using the fact
\begin{align*}
&D_2^\top\big\{I-\Si D_1\h R_1^{-1}D_1^\top-\Si[D_2-D_1\h R_1^{-1}D_1^\top P_1D_2]\cR^{-1}\nn\\
&\qq\times[D_2^\top-D_2^\top P_1D_1\h R_1^{-1}D_1^\top]\big\}\h\Si^{-1}\Si[D_2-D_1\h R_1^{-1}D_1^\top P_1D_2]\\
&\q=D_2^\top\Si[D_2-D_1\h R_1^{-1}D_1^\top P_1D_2],
\end{align*}
equality \rf{(a)-1} can be  simplified as follows:
\begin{align}
(a)
%
%
%
%
%
%
%
%
%
%
%
&=\big\{-D_2^\top P_1D_1\h R_1^{-1}D_1^\top+D_2^\top P_1D_1[R_1+D_1^\top (P_1-\Si)D_1]^{-1}D_1^\top\nn\\
\nn&\qq -D_2^\top P_1D_1[R_1+D_1^\top (P_1-\Si)D_1]^{-1}D_1^\top\Si D_1\h R_1^{-1}D_1^\top\\
\nn &\qq+D_2^\top\Si D_1\h R_1^{-1}D_1^\top\big\}\h\Si^{-1} \Si[D_2-D_1\h R_1^{-1}D_1^\top P_1D_2]\\
\nn&\q-D_2^\top \Si D_1[R_1+D_1^\top (P_1-\Si)D_1]^{-1}D_1^\top\Si [D_2-D_1\h R_1^{-1}D_1^\top P_1D_2]\\
%
%
\nn&\q+D_2^\top \Si D_1[R_1+D_1^\top (P_1-\Si)D_1]^{-1}D_1^\top\Si [D_1\h R_1^{-1}D_1^\top P_1D_2-D_2]\\
%
%
%
&\q\times\cR^{-1}[D_2^\top-D_2^\top P_1D_1\h R_1^{-1}D_1^\top]\h\Si^{-1} \Si[D_2-D_1\h R_1^{-1}D_1^\top P_1D_2].\label{lemma-proof1}
\end{align}
Then by substituting
\begin{align*}
&D_2^\top \Si D_1[R_1+D_1^\top (P_1-\Si)D_1]^{-1}D_1^\top\big\{I-\Si D_1\h R_1^{-1}D_1^\top-\Si[D_2-D_1\h R_1^{-1}D_1^\top P_1D_2]\nn\\
&\q\times\cR^{-1}[D_2^\top-D_2^\top P_1D_1\h R_1^{-1}D_1^\top]\big\}\h\Si^{-1}\Si[D_2-D_1\h R_1^{-1}D_1^\top P_1D_2]\\
&\q=D_2^\top \Si D_1[R_1+D_1^\top (P_1-\Si)D_1]^{-1}D_1^\top\Si[D_2-D_1\h R_1^{-1}D_1^\top P_1D_2]
\end{align*}
into \rf{lemma-proof1}, we get
\begin{align}
(a)
&=\big\{D_2^\top P_1D_1[R_1+D_1^\top (P_1-\Si)D_1]^{-1}D_1^\top-D_2^\top P_1D_1\h R_1^{-1}D_1^\top\nn\\
\nn&\qq -D_2^\top P_1D_1[R_1+D_1^\top (P_1-\Si)D_1]^{-1}D_1^\top\Si D_1\h R_1^{-1}D_1^\top+D_2^\top\Si D_1\h R_1^{-1}D_1^\top\\
\nn &\qq-D_2^\top \Si D_1[R_1+D_1^\top (P_1-\Si)D_1]^{-1}D_1^\top+D_2^\top \Si D_1[R_1+D_1^\top (P_1-\Si)D_1]^{-1}\\
\nn&\qq \times D_1^\top\Si D_1\h R_1^{-1}D_1^\top\big\} \h\Si^{-1} \Si[D_2-D_1\h R_1^{-1}D_1^\top P_1D_2]=0.
\end{align}
The proof is complete.
$\hfill\qed$

\subsection{Details in the proof of \autoref{Thm:Riccati}}
{\it\bf  Verification of $(f_3)_1=0$.}
By the definition \rf{def-hPhi} of $\h\Phi$, we have
\begin{align}
\cR(f_3)_1&= [I+ \cS_2\h\Si^{-1}\Si\cS_2^\top\cR^{-1}]D_2^\top(P_1-\Si)D_1[R_1+D_1^\top(P_1-\Si)D_1]^{-1}\nn\\
&\q-[I+ \cS_2\h\Si^{-1}\Si\cS_2^\top\cR^{-1}]D_2^\top P_1D_1\h R_1^{-1}+\cS_2\h\Si^{-1}\Si D_1\h R_1^{-1}.\label{f3-11}
\end{align}
It follows that
\begin{align}
&\cR(f_3)_1[R_1+D_1^\top(P_1-\Si)D_1]\nn\\
&\q=D_2^\top P_1D_1\h R_1^{-1}D_1^\top\Si D_1-D_2^\top\Si D_1-\cS_2\h\Si^{-1}\Si\cS_2^\top\cR^{-1}D_2^\top\Si D_1+\cS_2\h\Si^{-1}\Si D_1\nn\\
&\qq+\cS_2\h\Si^{-1}\Si\cS_2^\top\cR^{-1}D_2^\top P_1D_1\h R_1^{-1}D_1^\top\Si D_1-\cS_2\h\Si^{-1}\Si D_1\h R_1^{-1}D_1^\top\Si D_1\nn\\
&\q=D_2^\top P_1D_1\h R_1^{-1}D_1^\top\Si D_1-D_2^\top\Si D_1+\cS_2\h\Si^{-1}\h\Si\Si D_1=0.
\end{align}
Since $\cR\ll 0$ and $R_1+D_1^\top(P_1-\Si)D_1\gg 0$, the above implies that $(f_3)_1=0$.
$\hfill\qed$

\ms

\no
{\it\bf  Verification of $(f_3)_2=0$.}
By the fact $(f_3)_1=0$, we can simplify $(f_3)_2$ as follows:
\begin{align}
(f_3)_2&=\{\h\Phi D_2^\top-\h \Phi D_2^\top(P_1-\Si)D_1[R_1+D_1^\top(P_1-\Si)D_1]^{-1}D_1^\top-\cR^{-1}\cS_2\h\Si^{-1}\}\Si.\label{f3-21}
\end{align}
Then by the definition \rf{def-hPhi} of $\h\Phi$ and \rf{f3-11}, we have
\begin{align}
\cR(f_3)_2&=[I+\cS_2\h\Si^{-1}\Si\cS_2^\top\cR^{-1}]D_2^\top\Si-\cS_2\h\Si^{-1}\Si\nn\\
&\q- [I+\cS_2\h\Si^{-1}\Si\cS_2^\top\cR^{-1}]D_2^\top(P_1-\Si)D_1[R_1+D_1^\top(P_1-\Si)D_1]^{-1}D_1^\top\Si\nn\\
&=D_2^\top P_1D_1\h R_1^{-1}D_1^\top\Si- \cS_2\h\Si^{-1}\Si D_1\h R_1^{-1}D_1^\top\Si
+ \cS_2\h\Si^{-1}\Si\cS_2^\top\cR^{-1}D_2^\top P_1D_1\h R_1^{-1}D_1^\top\Si\nn\\
&\q- [I+\cS_2\h\Si^{-1}\Si\cS_2^\top\cR^{-1}]D_2^\top(P_1-\Si)D_1[R_1+D_1^\top(P_1-\Si)D_1]^{-1}D_1^\top\Si\nn\\
&=-\cR(f_3)_1D_1^\top\Si=0,
\end{align}
which implies $(f_3)_2=0$. $\hfill\qed$

\ms

\no
{\it\bf  Verification of $(f_1)_2=0$}. By the fact $(f_3)_1=0$,  we get
\begin{align}
&\h R_1 (f_1)_2[R_1+D_1^\top(P_1-\Si)D_1]\nn\\
&\q=D_1^\top\Si D_1+D_1^\top P_1 D_2\cR^{-1}\cS_2\h\Si^{-1}\Si D_1-D_1^\top P_1 D_2\cR^{-1}\cS_2\h\Si^{-1}\Si D_1\h R_1^{-1}D_1^\top\Si D_1\nn\\
&\qq-D_1^\top\h\Si^{-1}\Si D_1+D_1^\top\h\Si^{-1}\Si D_1\h R_1^{-1}D_1^\top\Si D_1-D_1^\top\h\Si^{-1}\Si\cS_2^\top\cR^{-1} D_2^\top P_1D_1\h R_1^{-1}D_1^\top\Si D_1\nn\\
&\qq+D_1^\top P_1D_2\h\Phi D_2^\top P_1D_1\h R_1^{-1}D_1^\top\Si D_1+D_1^\top \h\Si^{-1}\Si\cS_2^\top\cR^{-1}D_2^\top\Si D_1-D_1^\top P_1D_2\h\Phi D_2^\top\Si D_1.\nn
\end{align}
Then by the definitions of $\h\Phi$, $\h\Si$ and $\cS_2$, the above can be simplified as
\begin{align*}
&\h R_1 (f_1)_2[R_1+D_1^\top(P_1-\Si)D_1]\\
&\q=D_1^\top P_1 D_2\cR^{-1}\cS_2\h\Si^{-1}\Si D_1-D_1^\top P_1 D_2\cR^{-1}\cS_2\h\Si^{-1}\Si D_1\h R_1^{-1}D_1^\top\Si D_1\\
&\qq+D_1^\top P_1D_2\h\Phi D_2^\top P_1D_1\h R_1^{-1}D_1^\top\Si D_1-D_1^\top P_1D_2\h\Phi D_2^\top\Si D_1\\
&\q=D_1^\top P_1 D_2\cR^{-1}\cS_2\Si D_1+D_1^\top P_1 D_2\cR^{-1}D^\top_2P_1D_1\h R_1^{-1}D_1^\top\Si D_1
-D_1^\top P_1D_2\cR^{-1}D_2^\top\Si D_1\\
&\q=0.
\end{align*}
It follows that $(f_1)_2=0$. $\hfill\qed$

\ms

\no
{\it\bf  Verification of $(f_1)_3=0$}. By the definition of $\h\Phi$ and \rf{f3-21}, we get
\begin{align}
&[R_1+D_1^\top(P_1-\Si)D_1](f_1)_3\nn\\
&\q=D_1^\top(P_1-\Si)D_2\big\{\h\Phi D_2^\top-\h\Phi D_2^\top(P_1-\Si)D_1[R_1+D_1^\top(P_1-\Si)D_1]^{-1}D_1^\top-\cR^{-1}\cS_2\h\Si^{-1}\big\}\Si C\nn\\
&\q=D_1^\top(P_1-\Si)D_2(f_3)_2.
\end{align}
The result then follows from $(f_3)_2=0$. $\hfill\qed$

\ms

\no
{\it\bf  Verification of $(f_7)_5=0$}. Note that
\begin{align}
(f_7)_5&=C^\top \Si D_1\cM_{11} D_1^\top(P_1-\Si)C+C^\top \Si D_1\cM_{12} D_2^\top(P_1-\Si)C
+C^\top \Si D_2\cM_{21}D_1^\top(P_1-\Si)C\nn\\
&\q+C^\top \Si D_2\cM_{22} D_2^\top(P_1-\Si)C +C^\top(P_1-\Si) C-C^\top\h\Si^{-1}(P_1-\Si) C.
\end{align}
Thus, to prove $(f_7)_5=0$, it is sufficient to show that
\bel{(c5)-proof1}
(f_7)_{51}\deq \Si D_1\cM_{11} D_1^\top +\Si D_1\cM_{12} D_2^\top
+ \Si D_2\cM_{21}D_1^\top+ \Si D_2\cM_{22} D_2^\top +I-\h\Si^{-1}=0.
\ee
By the definition of $\cM_{ij};i,j=1,2$, we get
\begin{align}
\h\Si (f_7)_{51}
%
%
%
%
%
%
%
&=\Si\Big\{-\cS_2^\top\cR^{-1}\cS_2-\cS_2^\top\cR^{-1}\cS_2\Si D_1[R_1+D_1^\top(P_1-\Si)D_1]^{-1}D_1^\top\nn\\
&\q+\{D_1[R_1+D_1^\top(P_1-\Si)D_1]^{-1}D_1^\top(P_1-\Si)-I\}D_2\h\Phi D_2^\top\{(P_1-\Si)D_1\nn\\
&\q\times[R_1+D_1^\top(P_1-\Si)D_1]^{-1}D_1^\top-I\}+\{\cN-\cS_2^\top\cR^{-1}\cS_2\}\Si\nn\\
&\q\times\{D_1[R_1+D_1^\top(P_1-\Si)D_1]^{-1}D_1^\top(P_1-\Si)-I\}D_2^\top\h\Phi D_2^\top\nn\\
&\q\times\{(P_1-\Si)D_1[R_1+D_1^\top(P_1-\Si)D_1]^{-1}D_1^\top-I\}\Big\}\nn\\
&\deq \Si(f_7)_{52}.
\end{align}
By the definition of $\cN$, we get
\begin{align}
(f_7)_{52}
%
%
%
%
%
%
%
&=\big\{\cS_2^\top\cR^{-1}\cS_2\Si D_2\h\Phi D_2^\top-\cS_2^\top\cR^{-1}\cS_2\Si D_1 [R_1+D_1^\top(P_1-\Si)D_1]^{-1}D_1^\top(P_1-\Si)D_2\h\Phi D^\top_2\nn\\
&\qq-D_2^\top\h\Phi D_2^\top-\cN P_1D_1\h\Phi D_2^\top\big\}\{(P_1-\Si)D_1[R_1+D_1^\top(P_1-\Si)D_1]^{-1}D_1^\top-I\}\nn\\
&\q-\cS_2^\top\cR^{-1}\cS_2-\cS_2^\top\cR^{-1}\cS_2\Si D_1[R_1+D_1^\top(P_1-\Si)D_1]^{-1}D_1^\top,
\end{align}
which yields
\begin{align}
(f_7)_{52}
&=\cS_2^\top\cR^{-1}\cS_2\big\{\h\Si^{-1}\Si\cS_2^\top\cR^{-1}+\Si D_1[R_1+D_1^\top(P_1-\Si)D_1]^{-1}D_1^\top (P_1-\Si)D_2\h\Phi\nn\\
&\q-\Si D_2\h\Phi\big\}D^\top_2+\cS_2^\top\cR^{-1}D_2^\top P_1D_1\h R_1^{-1}D_1^\top-\cS_2^\top\cR^{-1}\cS_2\Si D_1[R_1+D_1^\top(P_1-\Si)D_1]^{-1}D_1^\top\nn\\
&\q-\cS_2^\top\big\{\h\Phi+\cR^{-1}\cS_2\Si D_1[R_1+D_1^\top(P_1-\Si)D_1]^{-1}D_1^\top (P_1-\Si)D_2\h\Phi -\cR^{-1}\cS_2\Si D_2\h\Phi\big\}\nn\\
&\q \times D_2^\top(P_1-\Si) D_1[R_1+D_1^\top(P_1-\Si)D_1]^{-1}D_1^\top\nn\\
&\deq \cS_2^\top\cR^{-1}\cS_2(f_7)_{53}D_2^\top+(f_7)_{54} D_1^\top.\label{f7-52}
\end{align}
By \rf{f3-21} and the fact $\Si-\h\Si\Si(\h\Si^{-1})^\top=0$, we get $(f_7)_{53}=0$. Moreover,
\begin{align}
(f_7)_{54}
&=\cS_2^\top\cR^{-1}D_2^\top P_1D_1\h R_1^{-1}-\cS_2^\top\cR^{-1}\cS_2\Si D_1[R_1+D_1^\top(P_1-\Si)D_1]^{-1}\nn\\
&\q-\cS_2^\top\cR^{-1}D_2^\top(P_1-\Si)D_1[R_1+D_1^\top(P_1-\Si)D_1]^{-1}=0.
\end{align}
Substituting $(f_7)_{53}=0$ and $(f_7)_{54}=0$ into \rf{f7-52} yields $(f_7)_{52}=0$, which then implies $(f_7)_5=0$.
$\hfill\qed$

\end{document}